\documentclass[12pt]{amsart}

\usepackage{amsfonts,latexsym,amsmath,amssymb,amsthm,mathrsfs}
\usepackage{graphicx}
\usepackage{tikz}
\usetikzlibrary{arrows.meta,calc}
\usepackage{needspace}
\usepackage{placeins}
\usepackage[labelsep=space,labelfont=up]{caption}

\usepackage[colorlinks=true,urlcolor=blue,linkcolor=blue,anchorcolor=blue,citecolor=blue]{hyperref}
\usepackage[nameinlink,capitalize]{cleveref}
\allowdisplaybreaks[4]

\makeatletter
\@namedef{subjclassname@2020}{\textup{2020} Mathematics Subject Classification}
\makeatother

\numberwithin{equation}{section}

\theoremstyle{plain}
\newtheorem{theorem}{Theorem}[section]
\newtheorem{lemma}[theorem]{Lemma}
\newtheorem{proposition}[theorem]{Proposition}
\newtheorem{corollary}[theorem]{Corollary}
\theoremstyle{definition}
\newtheorem{definition}[theorem]{Definition}
\theoremstyle{plain}
\newtheorem{example}[theorem]{Example}
\newtheorem{remark}[theorem]{Remark}

\theoremstyle{remark}
\newtheorem*{remark*}{Remark}

\newcommand{\R}{\mathbb R}
\newcommand{\dd}{\,d}
\newcommand{\Sn}{\mathbb S}
\newcommand{\sphereareanm}{|\Sn^{n-1}|}
\DeclareMathOperator{\Lip}{Lip}

\title{Alexandrov--Fenchel inequalities for closed axisymmetric hypersurfaces}

\author{Weimin Sheng}
\address{Weimin Sheng: School of Mathematical Sciences, Zhejiang University, Hangzhou 310058, China.}
\email{weimins@zju.edu.cn}

\author{Yinhang Wang}
\address{Yinhang Wang: School of Mathematical Sciences, Zhejiang University, Hangzhou 310058, China.}
\email{wangyinhang@zju.edu.cn}

\hypersetup{
  pdftitle={Alexandrov--Fenchel inequalities for closed axisymmetric hypersurfaces},
  pdfauthor={Weimin Sheng and Yinhang Wang},
  pdfsubject={Geometric inequalities for axisymmetric hypersurfaces},
  pdfkeywords={Alexandrov--Fenchel inequality, quermassintegral, higher-order mean curvature, axisymmetric hypersurface, monotone rearrangement, nonconvex hypersurface}
}

\begin{document}

\begin{abstract}
We establish sharp Alexandrov--Fenchel inequalities for closed,
embedded, axisymmetric hypersurfaces
$\Sigma\subset\mathbb R^{n+1}$, $n\ge2$, with admissible $C^{1,1}$ profiles.
 Let \(H_j\) be the $j$-th normalized mean curvature. If $\Sigma$ is axiconvex, i.e., every
hyperplane orthogonal to the symmetry axis meets the enclosed domain in an
$n$-dimensional ball or empty, then 
\[
\left(\frac{\int_\Sigma H_\ell\,dA}{|\mathbb S^n|}\right)^{1/(n-\ell)}
\le
\left(\frac{\int_\Sigma H_k\,dA}{|\mathbb S^n|}\right)^{1/(n-k)},
\qquad 0\le\ell<k\le n-1.
\]
Without axiconvexity, the inequalities hold for sign-corrected
integrals $\int_\Sigma\varepsilon_\theta^{\,j}H_j\,dA$, where
$\varepsilon_\theta=\operatorname{sgn}(\sin\theta)$ denotes the sign of the
azimuthal principal curvatures.  As applications, the Alexandrov--Fenchel inequalities remain valid whenever both \(\ell\) and \(k\) are even. We also prove
\[
\left(\frac{\int_\Sigma H_\ell\,dA}{|\mathbb S^n|}\right)^{1/(n-\ell)}
\le
\left(\frac{\int_\Sigma |H_k|\,dA}{|\mathbb S^n|}\right)^{1/(n-k)},
\qquad 0\le\ell<k\le n-1,\quad \ell\ \text{even}.
\]
Equality in each of the above inequalities holds if and only if $\Sigma$ is a round
sphere. The proof adapts the rearrangement and tangent angle folding approach of \cite{DHMT} from surfaces of revolution in \(\mathbb R^3\) to $n$-dimensional axisymmetric hypersurfaces in $\mathbb{R}^{n+1}$. The higher-dimensional construction uses weighted monotone rearrangement to produce convex comparison bodies with \(C^{1,1}\) boundaries.
\end{abstract}

\maketitle

\baselineskip16pt
\parskip1pt

\setcounter{tocdepth}{2}
\tableofcontents

\section{Introduction}

The Alexandrov--Fenchel inequalities
\cite{Alexandrov1937,Alexandrov1938} are among the fundamental inequalities in the Brunn--Minkowski theory. For a smooth convex body $\Omega\subset\R^{n+1}$ with boundary $\Sigma=\partial\Omega$, let
\[
 H_k=\binom nk^{-1}\sigma_k(\kappa_1,\ldots,\kappa_n)
 \qquad 0\le k\le n,
\]
be the normalized $k$-th mean curvature of $\Sigma$. The celebrated Alexandrov--Fenchel inequality states that
\begin{equation}\label{eq:intro-convex-AF-chain}
 \left(
   \frac{\int_\Sigma H_\ell\dd A}{|\Sn^n|}
 \right)^{1/(n-\ell)}
 \le
 \left(
   \frac{\int_\Sigma H_k\dd A}{|\Sn^n|}
 \right)^{1/(n-k)}
 \qquad(0\le\ell<k\le n-1).
\end{equation}
Equality in
\eqref{eq:intro-convex-AF-chain} holds if and only if $\Sigma$ is a round sphere. These include the Minkowski inequality
\begin{equation}\label{eq:intro-convex-minkowski}
  \int_\Sigma H_1 \dd A
  \ge
  |\Sn^n|^{1/n}|\Sigma|^{(n-1)/n}
\end{equation}
as a special case. We refer the reader to Schneider's book \cite{Schneider} for a detailed introduction. 

Extending \eqref{eq:intro-convex-AF-chain} beyond convexity is one of the important questions in classical geometry. Using the inverse curvature flow introduced by Gerhardt \cite{Gerhardt1990} and Urbas \cite{Urbas1990}, Guan-Li \cite{GuanLi} showed that the inequalities \eqref{eq:intro-convex-AF-chain} still hold for star-shaped and $k$-convex domains. Without star-shapedness, Chang-Wang \cite{ChangWang2013} obtained \eqref{eq:intro-convex-AF-chain} under $(k+1)$-convexity by optimal transport, although their constants were not sharp. For inequalities relating volume to the integral of $H_k$, Chang and Wang \cite{ChangWang} obtained sharp bounds
for $k=1,2$, which Qiu \cite{Qiu} extended to higher orders under
the same $(k+1)$-convexity assumption. Using the weak inverse mean curvature flow developed in \cite{HuiskenIlmanen2001}, Huisken \cite{Huisken2009} showed that the Minkowski inequality
\eqref{eq:intro-convex-minkowski} also holds for outward-minimizing domains, see also \cite{FreireSchwartz2014}. Agostiniani-Fogagnolo-Mazzieri \cite{AFM}
developed an alternative approach based on nonlinear potential
theory, avoiding the high-dimensional regularity difficulties
of the weak inverse mean curvature flow.

In this paper, we consider the validity of \eqref{eq:intro-convex-AF-chain} for axisymmetric hypersurfaces. Our investigation was inspired by the work \cite{DHMT}. In $\R^3$, Dalphin--Henrot--Masnou--Takahashi \cite{DHMT} proved the
Minkowski inequality \eqref{eq:intro-convex-minkowski} for axiconvex $C^{1,1}$ surfaces by rearranging
the tangent angle.
We extend this construction to higher dimensions and higher mean curvature integrals.

\begin{definition}\label{def:admissible-axisymmetric}
	Let $n\ge2$. Let $\Omega\subset\R^{n+1}$ be a bounded domain and 
	$\Sigma=\partial\Omega$ be connected, embedded and
	axisymmetric.  We call
	$\Sigma$ a member of the \emph{admissible $C^{1,1}$ axisymmetric class}
	if its meridian in the half-plane $\{x\ge0\}$ is an injective
	arclength curve
	\[
	\gamma(s)=(x(s),z(s))\in C^{1,1}([0,L];\R^2),
	\]
	oriented from the lower to the upper pole, with
	\[
	x(0)=x(L)=0,\qquad x(s)>0\quad\text{for }0<s<L,
	\qquad z(L)>z(0)=0.
	\]
	We require its unit tangent to have a Lipschitz angle lift
	$\theta\in W^{1,\infty}(0,L)$ satisfying
	\[
	\dot x=\cos\theta,\qquad \dot z=\sin\theta,\qquad
	\theta(0)=0,\qquad \theta(L)=\pi.
	\]
\end{definition}

In particular, rotational hypersurfaces generated by closed meridians disjoint from the symmetry axis are excluded.

An admissible hypersurface is {axiconvex} if
$0\le\theta\le\pi$,
equivalently, if the intersection of the enclosed domain with every hyperplane orthogonal to the symmetry axis is either an \(n\)-dimensional ball centered on the axis or empty.
Within the admissible class, the following inclusions are strict:
(convex and axisymmetric) $\subset$ axiconvex $\subset$ axisymmetric.
Moreover, \cref{ex:axiconvex-outside-gamma} gives an axiconvex
example satisfying $\sigma_k<0$ at the equator for every
$1\le k\le n$. Thus, the axiconvex class is not contained in
the class of star-shaped $k$-convex hypersurfaces for any such $k$.

Our first result establishes the Alexandrov--Fenchel inequalities
for closed axiconvex hypersurfaces.

\begin{theorem}
\label{thm:closed-high-dim}
Let $\Sigma^n\subset\R^{n+1}$ be an axiconvex hypersurface of the admissible
$C^{1,1}$ axisymmetric class.  Define
\[
 \mathcal M_j(\Sigma):=\int_\Sigma H_j\dd A,
 \qquad 0\le j\le n-1.
\]
Then, for $0\le\ell<k\le n-1$,
\begin{equation}\label{eq:axiconvex-curvature-inequality}
	 \left(\frac{\mathcal M_\ell(\Sigma)}{|\Sn^n|}\right)^{1/(n-\ell)}
	\le
	\left(\frac{\mathcal M_k(\Sigma)}{|\Sn^n|}\right)^{1/(n-k)}.
\end{equation}
Equality holds if and only if $\Sigma$
is a round sphere.
\end{theorem}

The following corollary is an immediate reformulation of the inequality \eqref{eq:axiconvex-curvature-inequality}, with the volume term supplied by the Euclidean
isoperimetric inequality.

\begin{corollary}
\label{cor:complete-W-chain}
Let $\Omega$ be the bounded domain enclosed by an axiconvex hypersurface $\Sigma$ in
\cref{thm:closed-high-dim}, and let
$\omega_{n+1}=|B^{n+1}|$.  Define
\[
 W_0(\Omega)=|\Omega|,
 \qquad
 W_i(\Omega)=\frac1{n+1}\mathcal M_{i-1}(\Sigma),
 \quad 1\le i\le n.
\]
Then, for $0\le i<j\le n$,
\begin{equation}\label{eq:W-chain}
	\left(\frac{W_i(\Omega)}{\omega_{n+1}}\right)^{1/(n+1-i)}
	\le
	\left(\frac{W_j(\Omega)}{\omega_{n+1}}\right)^{1/(n+1-j)}.
\end{equation}
Equality holds if and only if $\Omega$ is a ball.
\end{corollary}

The proof of \cref{thm:closed-high-dim} is based on a convex comparison principle.
We construct a convex body $K_k$ with $C^{1,1}$ boundary such that
\[
\mathcal M_k(\partial K_k)=\mathcal M_k(\Sigma),
\qquad
\mathcal M_\ell(\partial K_k)\ge\mathcal M_\ell(\Sigma)
\quad(0\le\ell<k).
\]
Thus the convex body preserves the prescribed curvature integral and
simultaneously increases every lower-order integral.  The classical
inequalities for $K_k$ then give the inequalities on $\Sigma$.

Extending the construction of \cite{DHMT} to higher dimensions involves two main difficulties. The curvature formula naturally selects the weighted parameter \mbox{$\dd\tau=x^{n-k-1}\dd s$}, in which $\mathcal M_k$ 
depends only on the distribution of the tangent angle and is
therefore preserved by increasing rearrangement.  The weight is
trivial when $k=n-1$, as in the three-dimensional case, but is otherwise nonconstant.  First, preserving $\mathcal M_k$ does not directly yield the comparisons of $\mathcal M_\ell$.  In the weighted coordinate, $\mathcal M_\ell$ contains an additional factor $x^{k-\ell}$,
which becomes a power of the cosine primitive.
For $\ell\ge1$, this factor is coupled to an angle-dependent coefficient.  Since rearrangement changes both factors,
the increase of the primitive alone, which suffices for the area
comparison, does not yield the required comparisons of $\mathcal M_\ell$.  We therefore establish a primitive rearrangement inequality adapted to these coefficients.  
Second, the weighted change of variables degenerates at the poles
when $k<n-1$.  Although the original angle is Lipschitz in arclength,
it need not be Lipschitz in $\tau$, so Lipschitz preservation
under rearrangement cannot be applied directly.
We instead estimate the growth of the weighted angle distribution
by the coarea formula and compare the reconstructed radius with
the radii on the original level sets.  This comparison compensates
for the singular factor introduced when arclength is restored.  Thus we recover Lipschitz control in the reconstructed arclength. This yields a comparison body
with $C^{1,1}$ boundary.  Its curvature
integrals therefore have the classical quermassintegral interpretation
needed to apply the convex inequalities.

Our second result removes the axiconvexity assumption.  Following the folding construction of \cite{DHMT}, set
$\bar\theta=\arccos(\cos\theta)$ and define
\[
\begin{gathered}
	\varepsilon_\theta:=\operatorname{sgn}(\sin\theta),
	\qquad \operatorname{sgn}(0)=0,\\
	\widehat H_0:=1,
	\qquad \widehat H_j:=\varepsilon_\theta^jH_j
	\quad(1\le j\le n).
\end{gathered}
\]
The resulting axiconvex hypersurface $\bar\Sigma$ satisfies
$\bar H_j=\widehat H_j$ almost everywhere and $d\bar A=dA$, so the curvature integrals of $\bar\Sigma$
coincide with the sign-corrected integrals of $\Sigma$.
Applying \cref{thm:closed-high-dim} therefore gives the
following extension to the full admissible class.

\begin{theorem}
\label{thm:folded-homogeneous-chain}
Let $\Sigma^n\subset\R^{n+1}$ be an admissible $C^{1,1}$ axisymmetric hypersurface.  Define
\begin{equation*}
 \widehat{\mathcal M}_j(\Sigma)
 :=\int_\Sigma \widehat H_j\dd A,
 \quad 0\le j\le n-1.
\end{equation*}
Then $\widehat{\mathcal M}_j(\Sigma)>0$,
and for $0\le\ell<k\le n-1$,
\begin{equation}\label{eq:folded-complete-chain}
 \left(
   \frac{\widehat{\mathcal M}_\ell(\Sigma)}{|\Sn^n|}
 \right)^{1/(n-\ell)}
 \le
 \left(
   \frac{\widehat{\mathcal M}_k(\Sigma)}{|\Sn^n|}
 \right)^{1/(n-k)}.
\end{equation}
Equality holds if and only if $\Sigma$ is a
round sphere.
\end{theorem}

For even orders, $\widehat H_j=H_j$ almost everywhere.
\Cref{thm:folded-homogeneous-chain} yields sharp inequalities
between curvature integrals of even orders,
without assuming axiconvexity.

\begin{corollary}
\label{cor:even-order-chain}
Let $\Sigma^n\subset\R^{n+1}$ be an admissible $C^{1,1}$ axisymmetric hypersurface.
Then, for every pair of even integers $0\le\ell<k\le n-1$,
\begin{equation}\label{eq:even-order-curvature-chain}
 \left(\frac{\mathcal M_\ell(\Sigma)}{|\Sn^n|}\right)^{1/(n-\ell)}
 \le
 \left(\frac{\mathcal M_k(\Sigma)}{|\Sn^n|}\right)^{1/(n-k)}.
\end{equation}
Moreover, for every even $\ell$ and every $k$ with
$0\le\ell<k\le n-1$,
\begin{equation}\label{eq:even-absolute-chain}
	\left(\frac{\int_\Sigma H_\ell\dd A}{|\Sn^n|}\right)^{1/(n-\ell)}
	\le
	\left(\frac{\int_\Sigma|H_k|\dd A}{|\Sn^n|}\right)^{1/(n-k)}.
\end{equation}
Equality in \eqref{eq:even-order-curvature-chain} or \eqref{eq:even-absolute-chain} holds if and only if $\Sigma$ is a
round sphere.
\end{corollary}

It is worth noting that Dalphin--Henrot--Masnou--Takahashi \cite[Theorem~1.2]{DHMT}
constructed embedded axisymmetric $C^{1,1}$ surfaces
of fixed area whose total mean curvature tends to $0^+$ in $\R^3$. Chodosh--Eichmair--Koerber \cite{CEK}
further constructed smooth axisymmetric counterexamples to
the Minkowski inequality that converge to a sphere
in $C^1\cap W^{2,p}$ for every $1\le p<\infty$.
Thus, axisymmetry does not ensure the Minkowski
inequality, even close to a sphere in these
topologies. More generally, Li--Wan--Zhang \cite[Theorem~1.2(i)]{LWZ}
constructed admissible axisymmetric hypersurfaces of fixed
area whose total $j$-th mean curvature tends to $0$
for every odd $1\le j\le n-1$.
Thus the corresponding area inequality cannot hold throughout
the admissible class with the uncorrected odd-order curvature integral. By contrast, \cref{cor:even-order-chain} establishes the corresponding even-order inequalities in higher dimensions for every axisymmetric hypersurface in the admissible class.

Finally, taking $\ell=0$ in \eqref{eq:even-absolute-chain} gives
a sharp bound for area in terms of $\int_\Sigma |H_k|\dd A$, recovering
\cite[Theorem~1.3]{DHMT} when $n=2$ and $k=1$.
For $k=1$, an estimate without symmetry assumptions follows
from the Michael--Simon inequality \cite{MichaelSimon}, with
a nonsharp constant. Brendle \cite{Brendle2021} established a sharp Sobolev inequality
for hypersurfaces, possibly with boundary, using the ABP method. Its sharp constant is determined
by flat disks and does not give the optimal spherical constant in our setting. Glaudo \cite{Glaudo} obtained a sharp estimate
involving the positive part $H_1^+:=\max\{H_1,0\}$ for axisymmetric hypersurfaces sufficiently
$C^1$-close to a sphere. Wang--Yang \cite{WangYang} established
related stability estimates at higher orders for nearly
spherical axisymmetric hypersurfaces with negative-part
curvature corrections. 
Recently, Li--Wan--Zhang \cite[Theorems~1.1 and~1.3]{LWZ}
established the endpoint case $\ell=0$, $k=n-1$ of
\cref{thm:closed-high-dim,cor:even-order-chain}.
Their construction uses unweighted arclength rearrangement,
whereas ours combines weighted angle rearrangement with
arclength reconstruction to treat every $1\le k\le n-1$.

The paper is organized as follows. \Cref{sec:prelim} presents the geometric preliminaries and
rearrangement tools.
\Cref{sec:weighted-form} derives the weighted one-dimensional
representations of the curvature functionals.
\Cref{sec:rearrangement} constructs the convex comparison body,
establishes its $C^{1,1}$ regularity, and proves the primitive
rearrangement and curvature comparison inequalities.
Finally, \cref{sec:proofs} proves the main results and uses tangent-angle folding to obtain the sign-corrected
chain and its even-order consequences.

\section{Preliminaries}\label{sec:prelim}
Let $\Sigma=\partial\Omega$ be admissible in the sense of
\cref{def:admissible-axisymmetric}, with arclength profile $\gamma(s)=(x(s),z(s))$ and tangent angle
$\theta\in W^{1,\infty}(0,L)$.  Thus $\dot x=\cos\theta$, $\dot z=\sin\theta$.  
The axisymmetric parametrization
\[
(s,\omega)\longmapsto(x(s)\omega,z(s)),
\qquad \omega\in\Sn^{n-1},
\]
induces the metric $\dd s^2+x(s)^2g_{\Sn^{n-1}}$.  Hence
\begin{equation}\label{eq:area-element-prelim}
	\dd A=x^{n-1}\dd\omega\dd s,
	\qquad
	|\Sigma|=\sphereareanm\int_0^L x^{n-1}\dd s .
\end{equation}
We use
the outward unit normal $\nu=(\sin\theta\,\omega,-\cos\theta)$ and the
Weingarten convention $A=D\nu$, so that convex principal
curvatures are nonnegative. For almost every $s\in(0,L)$, the principal curvatures are
\[
  \kappa_1=\dot\theta,\qquad
  \kappa_2=\cdots=\kappa_n=\frac{\sin\theta}{x}.
\]
We use the normalized
elementary symmetric curvatures
\[
  H_k=\binom nk^{-1}\sigma_k(\kappa_1,\ldots,\kappa_n),
  \qquad 0\le k\le n,
\]
so that $H_0=1$ and
$H_1=\frac1n(\kappa_1+\cdots+\kappa_n)$.
Consequently, for $1\le k \le n$,
\begin{equation}\label{eq:H-integrand-prelim}
  H_k\dd A
  =
  \frac1n\left(
    k\dot\theta\,\sin^{k-1}\theta\,x^{n-k}
    +(n-k)\sin^k\theta\,x^{n-k-1}
  \right)\dd\omega\dd s .
\end{equation}

\begin{definition}\label{def:axiconvex}
	An admissible hypersurface $\Sigma$ as in
	\cref{def:admissible-axisymmetric} is called \emph{axiconvex}
	if its tangent angle satisfies
	$0\le\theta(s)\le\pi$ for all $s\in[0,L]$.
	Equivalently, its height function $z$ is nondecreasing.
	Within this class, $\Sigma$ is convex if and only if
	$\theta$ is nondecreasing.
\end{definition}

The following example shows that axiconvexity implies neither
convexity nor the standard G{\aa}rding-cone conditions.
In particular, the axiconvex class is not contained in the class
of star-shaped $k$-convex hypersurfaces.

\begin{example}
\label{ex:axiconvex-outside-gamma}
For $n\ge2$, define 
\begin{equation}\label{eq:explicit-non-k-convex-profile}
 x_n(t)=\sin t-\frac12\sin^{4n-1}t,
 \qquad
 z_n(t)=1-\cos t,
 \qquad
0\le t\le\pi.
\end{equation}
Consider the axisymmetric hypersurface parametrized by
\[
 (t,\omega)\longmapsto(x_n(t)\omega,z_n(t)),
 \qquad \omega\in\Sn^{n-1}.
\]
Since $x_n\ge\frac12\sin t$ and $z_n'=\sin t>0$ on $(0,\pi)$,
the profile is injective.
Put
\[
 v(t)=\sqrt{x_n'(t)^2+z_n'(t)^2}.
\]
If $\theta$ denotes its tangent angle, then
\[
 \sin\theta=\frac{z_n'}{v}>0\quad(0<t<\pi),
 \qquad \theta(0)=0,\qquad\theta(\pi)=\pi.
\]
Thus the resulting closed embedded axisymmetric hypersurface is
axiconvex.
 At $t=\pi/2$, the principal
curvatures are
\[
\kappa_1=-\frac{4n-3}{2},\qquad \kappa_2=\cdots=\kappa_n=2.
\]
For every $1\le k\le n$,
\[
-\frac{4n-3}{2}+\frac{2(n-k)}{k}
\le -\frac{4n-3}{2}+2(n-1)
=-\frac12<0.
\]
Consequently, 
\[
\sigma_k=\binom{n-1}{k-1}2^{k-1}
\left(-\frac{4n-3}{2}+\frac{2(n-k)}{k}\right)<0.
\]
Thus its principal curvature vector lies outside every
$\overline\Gamma_k$ at its equator, where $\Gamma_k =\{\kappa\in\R^n:\sigma_j (\kappa)>0, 1\le j\le k\}$ and $\overline\Gamma_k$ denotes its closure.
\end{example}

The admissible axiconvex meridians of
\cref{def:admissible-axisymmetric} can equivalently be described
by an angle $\theta\in W^{1,\infty}(0,L)$ taking values in
$[0,\pi]$ and satisfying
\begin{equation}\label{eq:admissible-endpoints}
	\theta(0)=0,\qquad \theta(L)=\pi,
\end{equation}
\begin{equation}\label{eq:admissible-closure}
	\int_0^L\cos\theta(t)\,dt=0,
	\qquad
	\int_0^L\sin\theta(t)\,dt>0,
\end{equation}
\begin{equation}\label{eq:admissible-positive-radius}
	\int_0^s\cos\theta(t)\,dt>0
	\qquad\text{for every }0<s<L.
\end{equation}
Indeed, the corresponding meridian is
\[
\gamma(s)=(x(s),z(s))
=\left(
\int_0^s\cos\theta(t)\,dt,\,
\int_0^s\sin\theta(t)\,dt
\right).
\]
Thus \eqref{eq:admissible-closure} and
\eqref{eq:admissible-positive-radius} express precisely
$x(0)=x(L)=0$, $x>0$ on $(0,L)$, and $z(L)>z(0)=0$.
Conversely, these formulae reconstruct a $C^{1,1}$ meridian
with unit speed, and its injectivity follows from
\cite[Proposition~3.1]{DHMT}.
The axisymmetric parametrization and tangent angle convention are
summarized in \cref{fig:high-dim-axiconvex}.

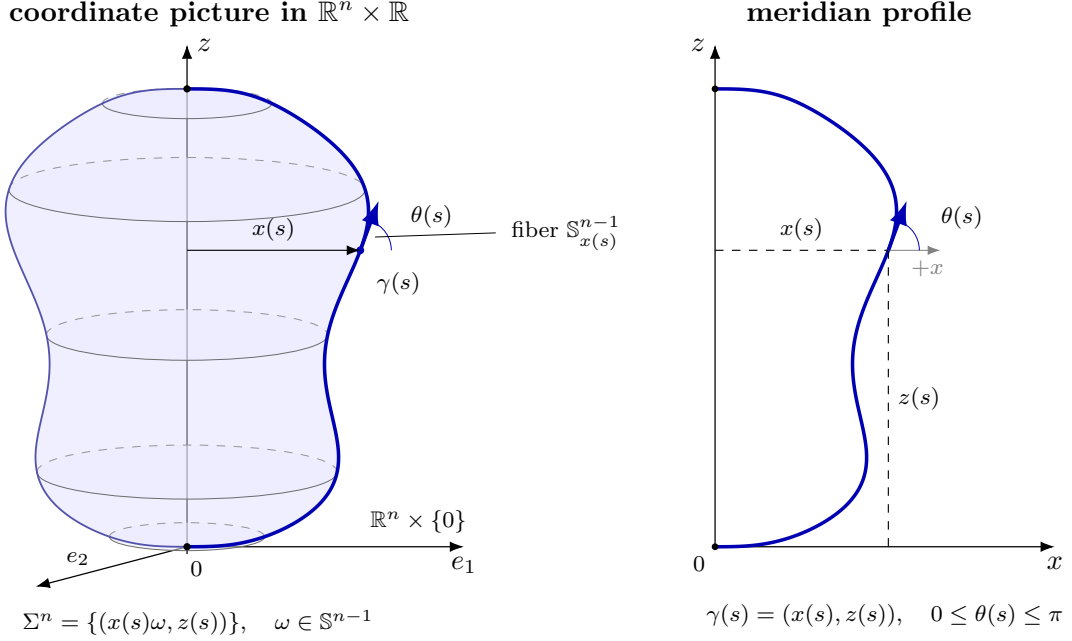
\begin{figure}[!t]
  \centering
  \begingroup
  \fontsize{10}{12}\selectfont
  \renewcommand{\footnotesize}{\fontsize{8}{9.5}\selectfont}
  \setlength{\fboxsep}{4pt}
  \setlength{\fboxrule}{0pt}
  \resizebox{0.92\textwidth}{!}{\fbox{%
\begin{tikzpicture}[
  >=Latex,
  scale=0.92,
  axis/.style={->,line width=0.45pt,black},
  profile/.style={line width=1.25pt,blue!70!black},
  backprofile/.style={line width=0.65pt,blue!35!gray},
  fiberfront/.style={draw=black!58,line width=0.35pt},
  fiberback/.style={draw=black!42,dashed,line width=0.30pt},
  note/.style={font=\footnotesize},
  declare function={
    xr(\u)=2.40*sin(180*\u)*(1+0.22*sin(180*\u)*sin(180*\u)*sin(720*\u-55));
    zr(\u)=6.20*(10*\u*\u*\u-15*\u*\u*\u*\u+6*\u*\u*\u*\u*\u);
    dxr(\u)=2.40*(pi*cos(180*\u)*(1+0.22*sin(180*\u)*sin(180*\u)*sin(720*\u-55))
      +sin(180*\u)*0.22*(2*pi*sin(180*\u)*cos(180*\u)*sin(720*\u-55)
      +4*pi*sin(180*\u)*sin(180*\u)*cos(720*\u-55)));
    dzr(\u)=186.0*\u*\u*(1-\u)*(1-\u);
  }
]

\node[font=\bfseries] at (0.30,7.22)
  {coordinate picture in $\mathbb R^n\times\mathbb R$};

\draw[axis] (0,0) -- (3.75,0) node[below] {$e_1$};
\draw[axis] (0,0) -- (-2.05,-0.54) node[pos=0.72,above,note] {$e_2$};
\node[note,anchor=west] at (2.35,0.34) {$\mathbb R^n\times\{0\}$};
\draw[axis] (0,-0.10) -- (0,6.80) node[right] {$z$};
\node[note] at (0.13,-0.30) {$0$};

\path[fill=blue!14,opacity=0.44]
  plot[smooth,variable=\t,domain=0:1,samples=160] ({xr(\t)},{zr(\t)})
  --
  plot[smooth,variable=\t,domain=1:0,samples=160] ({-xr(\t)},{zr(\t)})
  -- cycle;

\foreach \tt in {0.14,0.30,0.48,0.66,0.84} {
  \pgfmathsetmacro{\rx}{xr(\tt)}
  \pgfmathsetmacro{\zz}{zr(\tt)}
  \pgfmathsetmacro{\ry}{0.18*\rx}
  \path[fill=blue!8,opacity=0.50] (0,\zz) ellipse [x radius=\rx,y radius=\ry];
  \draw[fiberback] (-\rx,\zz) arc [start angle=180,end angle=0,
    x radius=\rx,y radius=\ry];
  \draw[fiberfront] (\rx,\zz) arc [start angle=0,end angle=-180,
    x radius=\rx,y radius=\ry];
}

\draw[backprofile]
  plot[smooth,variable=\t,domain=0:1,samples=160] ({-xr(\t)},{zr(\t)});
\draw[profile]
  plot[smooth,variable=\t,domain=0:1,samples=160] ({xr(\t)},{zr(\t)});
\fill (0,0) circle (1.4pt);
\fill (0,6.20) circle (1.4pt);

\pgfmathsetmacro{\tp}{0.58}
\pgfmathsetmacro{\px}{xr(\tp)}
\pgfmathsetmacro{\pz}{zr(\tp)}
\pgfmathsetmacro{\angp}{atan2(dzr(\tp),dxr(\tp))}
\coordinate (P) at (\px,\pz);
\fill[blue!70!black] (P) circle (1.5pt);
\draw[->,line width=0.4pt] (0,\pz) -- (P)
  node[midway,above,note] {$x(s)$};
\draw[->,profile] (P) -- ++({0.72*cos(\angp)},{0.72*sin(\angp)})
  ;
\draw[blue!70!black,line width=0.35pt]
  ($(P)+(0.42,0)$) arc [start angle=0,end angle=\angp,radius=0.42];
\node[note,anchor=west] at ($(P)+(0.55,0.46)$) {$\theta(s)$};
\node[note] at ($(P)+(0.52,-0.45)$) {$\gamma(s)$};
\draw[line width=0.35pt] ($(P)+(0.20,0.18)$) -- (4.18,4.25);
\node[anchor=west,note] at (4.26,4.25) {fiber $\mathbb S^{n-1}_{x(s)}$};

\node[note,align=center] at (0.15,-1.03)
  {$\Sigma^n=\{(x(s)\omega,z(s))\},\quad \omega\in\mathbb S^{n-1}$};

\begin{scope}[xshift=7.15cm]
  \node[font=\bfseries] at (1.95,7.22) {meridian profile};
  \draw[axis] (0,0) -- (0,6.80) node[left] {$z$};
  \draw[axis] (0,0) -- (4.60,0) node[below] {$x$};

  \draw[profile]
    plot[smooth,variable=\t,domain=0:1,samples=160] ({xr(\t)},{zr(\t)});

  \fill (0,0) circle (1.3pt) node[below left,note] {$0$};
  \fill (0,6.20) circle (1.3pt);

  \pgfmathsetmacro{\tq}{0.58}
  \pgfmathsetmacro{\qx}{xr(\tq)}
  \pgfmathsetmacro{\qz}{zr(\tq)}
  \pgfmathsetmacro{\angq}{atan2(dzr(\tq),dxr(\tq))}
  \coordinate (Q) at (\qx,\qz);
  \draw[dashed] (Q) -- (\qx,0);
  \draw[dashed] (0,\qz) -- (Q);
  \node[note,above] at (0.50*\qx,\qz) {$x(s)$};
  \node[note,right] at (\qx,0.50*\qz) {$z(s)$};
  \draw[->,black!50,line width=0.30pt] (Q) -- ++(0.70,0)
    node[pos=0.76,below,note] {$+x$};
  \draw[->,profile] (Q) -- ++({0.72*cos(\angq)},{0.72*sin(\angq)})
    ;
  \draw[blue!70!black,line width=0.35pt]
    ($(Q)+(0.42,0)$) arc [start angle=0,end angle=\angq,radius=0.42];
  \node[note,anchor=west] at ($(Q)+(0.58,0.45)$) {$\theta(s)$};

  \node[note,anchor=west] at (-0.25,-0.93)
    {$\gamma(s)=(x(s),z(s)),\quad 0\le\theta(s)\le\pi$};
\end{scope}

\end{tikzpicture}%
  }}
  \endgroup
  \caption{The axisymmetric axiconvex hypersurface $\Sigma$ and its generating meridian $\gamma(s)$}
  \label{fig:high-dim-axiconvex}
\end{figure}

We next relate curvature integrals to the classical convex quermassintegrals.
Write $\mathcal M_k(\Sigma)=\int_\Sigma H_k\dd A$, with $H_0=1$.
For a convex body
$K\subset\R^{n+1}$ with nonempty interior, define the
quermassintegrals by the Steiner formula
\[
|K+tB^{n+1}|=\sum_{j=0}^{n+1}\binom{n+1}{j}W_j(K)t^j.
\]
Here $B^{n+1}$ is the unit ball.

\begin{proposition}
\label{prop:C11-Steiner-compatibility}
If $K\subset\R^{n+1}$ is a convex body with $C^{1,1}$ boundary, then
for every $0\le k\le n$,
\begin{equation}\label{eq:C11-convex-Steiner-identification}
 (n+1)W_{k+1}(K)=\int_{\partial K}H_k\dd A.
\end{equation}
\end{proposition}

\begin{proof}
If $\partial K$ is smooth, expansion of the normal Jacobian and
integration in the normal distance give
\[
  |K+tB^{n+1}|
  =
  |K|+
  \sum_{k=0}^n
    \frac{t^{k+1}}{k+1}\binom nk
    \int_{\partial K}H_k\dd A .
\]
Since
\[
 \binom{n+1}{k+1}
 =\frac{n+1}{k+1}\binom nk,
\]
comparison with the Steiner formula shows that
\[
 (n+1)W_{k+1}(K)=\int_{\partial K}H_k\dd A .
\]
The same identity holds for a convex body with $C^{1,1}$ boundary.
Indeed, $\partial K$ is a closed $C^{1,1}$ hypersurface and therefore has positive reach by
\cite[Theorem~1.3]{Scholtes}; see also
\cite[Remark~4.20]{Federer1959}.  The almost everywhere curvature
representation of the Steiner coefficients in
\cite[Corollary~5.3]{Federer1959} (see also
\cite[Lemma~A.1]{Scholtes}) then gives
\eqref{eq:C11-convex-Steiner-identification}.
\end{proof}

For a bounded domain $\Omega$ with compact $C^{1,1}$ boundary, define
its \emph{curvature quermassintegrals} by
\[
 W_0(\Omega)=|\Omega|,
 \qquad
 W_{k+1}(\Omega):=\frac1{n+1}\mathcal M_k(\partial\Omega)
\]
for $0\le k\le n$.  By \cref{prop:C11-Steiner-compatibility}, these
agree with the classical quermassintegrals when $\Omega$ is convex.
For convex bodies, the Alexandrov--Fenchel inequalities, in the quermassintegral form
\cite[Theorem~7.3.1 and (7.67)]{Schneider}, give 
\begin{equation}\label{eq:convex-AF-prelim}
 \left(\frac{\mathcal M_\ell(\partial K)}{|\Sn^n|}\right)^{1/(n-\ell)}
 \le
 \left(\frac{\mathcal M_k(\partial K)}{|\Sn^n|}\right)^{1/(n-k)}
 \qquad(0\le\ell<k\le n-1).
\end{equation}
Equality in \eqref{eq:convex-AF-prelim} holds only
for a ball; see \cite[Section~7.6]{Schneider}.  

Finally, we collect the classical rearrangement properties needed below;
see \cite[Sections~1.1--1.2]{Kesavan} and
\cite[Section~II.2]{Kawohl}.

\begin{proposition}
\label{prop:monotone-rearrangements}
Let $T>0$ and let $f,g\in L^\infty(0,T)$ be real-valued.  Write
$a_f=\operatorname*{ess\,inf}f$ and
$b_f=\operatorname*{ess\,sup}f$, and define, for $0\le\tau\le T$,
\begin{align}
 f^\downarrow(\tau)
 &:=\inf\bigl\{a\in[a_f,b_f]:
       |\{t\in(0,T):f(t)>a\}|\le\tau\bigr\},
 \label{eq:decreasing-rearrangement-definition}\\
 f^\uparrow(\tau)
 &:=\sup\bigl\{a\in[a_f,b_f]:
       |\{t\in(0,T):f(t)\ge a\}|\ge T-\tau\bigr\}.
 \label{eq:increasing-rearrangement-definition}
\end{align}
Then $f^\uparrow=-(-f)^\downarrow$, and the following hold.
\begin{enumerate}
\item
The functions $f^\downarrow$ and $f^\uparrow$ are respectively
nonincreasing and nondecreasing,
and equimeasurable with $f$:
\[
 |\{f^\downarrow>a\}|=|\{f^\uparrow>a\}|=|\{f>a\}|
 \qquad(a\in\R).
\]
For every Borel function $\Psi$ with $\Psi(f)\in L^1(0,T)$,
\begin{equation}
 \int_0^T\Psi(f)\dd t
 =\int_0^T\Psi(f^\downarrow)\dd\tau
 =\int_0^T\Psi(f^\uparrow)\dd\tau.
 \label{eq:rearrangement-functional-invariance}
\end{equation}

\item
The primitives satisfy
\begin{equation}
 \int_0^\tau f^\uparrow(r)\dd r
 \le\int_0^\tau f(r)\dd r
 \le\int_0^\tau f^\downarrow(r)\dd r
 \qquad(0\le\tau\le T),
 \label{eq:HL}
\end{equation}
with equality at $\tau=T$.

\item
Monotone rearrangement is an $L^1$ contraction:
\begin{equation}
 \begin{aligned}
 \|f^\downarrow-g^\downarrow\|_{L^1(0,T)}
 &\le\|f-g\|_{L^1(0,T)},\\
 \|f^\uparrow-g^\uparrow\|_{L^1(0,T)}
 &\le\|f-g\|_{L^1(0,T)}.
 \end{aligned}
 \label{eq:rearrangement-L1-contraction}
\end{equation}

\item
If $f\in C([0,T])$, then $f^\downarrow,f^\uparrow\in C([0,T])$, with
\[
 f^\uparrow(0)=f^\downarrow(T)=\min_{[0,T]}f,
 \qquad
 f^\uparrow(T)=f^\downarrow(0)=\max_{[0,T]}f.
\]
\end{enumerate}
\end{proposition}

\begin{proof}
Part (1) follows from \cite[Proposition~1.1.3 and Corollary~1.1.3]{Kesavan}.
Applying \cite[Proposition~1.2.2]{Kesavan} to $f$ and $-f$, with
$E=(0,\tau)$, gives (2).
Part (3) is \cite[Proposition~1.2.1]{Kesavan}, with the
increasing case obtained by replacing $(f,g)$ with $(-f,-g)$.
If $f\in C([0,T])$, then
\[
 f([0,T])=[a_f,b_f],
 \qquad |\{a<f<b\}|>0\quad(a_f<a<b<b_f).
\]
A jump at $0<\tau\le T$ would allow
$f^\uparrow(\tau-)<a<b<f^\uparrow(\tau)$, whence
\[
 0=|\{a<f^\uparrow<b\}|=|\{a<f<b\}|>0.
\]
Thus $f^\uparrow\in C([0,T])$, including at $0$ by right-continuity.
The decreasing case follows by replacing $f$ with $-f$.
\end{proof}

\section{Weighted one-dimensional formulae}\label{sec:weighted-form}

Throughout this section and \cref{sec:rearrangement}, let $\Sigma$
be an admissible $C^{1,1}$ axiconvex hypersurface with arclength
profile $\gamma=(x,z)$ and tangent angle $\theta$, as in
\cref{def:admissible-axisymmetric,sec:prelim}.
Fix $k\in\{1,\ldots,n-1\}$, set $m=n-k$.
A dot denotes differentiation with respect to the original
arclength $s$, while a prime denotes differentiation with
respect to the indicated variable.
Define
\begin{equation}\label{eq:Jj-Fj}
	\begin{aligned}
		J_j(u)&:=\int_0^u\sin^{j-1}v\dd v
		&&(1\le j\le n-1),\\
		F_j(u)&:=\sin^ju-jJ_j(u)\cos u
		&&(1\le j\le n-1),
	\end{aligned}
\end{equation}
and set $F_0\equiv1$. Then
\begin{equation}\label{eq:Fj-monotone}
	F_j(0)=0,\qquad
	F_j'(u)=jJ_j(u)\sin u\ge0
	\quad(1\le j\le n-1,\ 0\le u\le\pi).
\end{equation}

The following lemma establishes the differential
identities for the weighted change of variables, which justify the subsequent calculations in the
$\tau$-coordinate.

\begin{lemma}
\label{lem:weighted-coordinate-calculus}
Let
\[
  \tau(s)=\int_0^s x(t)^{m-1}\dd t,
  \qquad
  T:=\tau(L)=\int_0^L x(t)^{m-1}\dd t,
\]
and write $s=s(\tau)$ for the inverse of the weighted coordinate.  For
$0\le\tau\le T$, set
\[
  X(\tau)=x(s(\tau)),
  \qquad
  Y(\tau)=\frac{X(\tau)^m}{m},
  \qquad
  \vartheta(\tau)=\theta(s(\tau)).
\]
Then
$
 s(\cdot),X\in \operatorname{C}([0,T])\cap C^1((0,T)),
 $
$
 Y\in C^1([0,T])
$
and
$
 \vartheta\in \operatorname{C}([0,T]).
$
Their derivatives satisfy
\[
 s'=X^{1-m},
 \qquad
 X'=X^{1-m}\cos\vartheta,
\qquad
 Y'=\cos\vartheta.
\]
In particular, when $m=1$, one has
$T=L$, $\tau(s)=s$, $s(\tau)=\tau$, $X=Y=x$, and
$\vartheta=\theta$.
\end{lemma}

\begin{proof}
Since $\tau\in C^1([0,L])$ and $\dot\tau=x^{m-1}>0$ on $(0,L)$,
it is strictly increasing and has a continuous inverse.  Thus $X$
and $\vartheta$ are continuous on $[0,T]$.  The inverse-function
theorem gives
\[
 s'=X^{1-m},\qquad
 X'=X^{1-m}\cos\vartheta,\qquad
 Y'=X^{m-1}X'=\cos\vartheta.
\]
Integrating the last identity and passing to the endpoints yields
\[
 Y(\tau)=\int_0^\tau\cos\vartheta(r)\dd r
 \qquad(0\le\tau\le T).
\]
The integrand is continuous, so $Y\in C^1([0,T])$.  Integrating $s'$ also gives
\[
 \int_0^T X^{1-m}\dd\tau
 =\lim_{\substack{a\rightarrow0\\ b\rightarrow T}}
   \bigl(s(b)-s(a)\bigr)=L.
 \qedhere
\]
\end{proof}

The weighted change of variables yields the following
one-dimensional representation of the total $k$-th mean curvature.

\begin{proposition}
\label{prop:weighted-form}
With the notation of \cref{lem:weighted-coordinate-calculus},
\begin{align}
  \int_\Sigma H_k\dd A
  &=
  \frac{\sphereareanm(n-k)}{n}
  \int_0^T F_k(\vartheta(\tau))\dd\tau.
  \label{eq:body-tau}
\end{align}
\end{proposition}

\begin{proof}
Integrating \eqref{eq:H-integrand-prelim} over the spherical orbits
gives
\begin{equation}\label{eq:H-formula}
 \int_\Sigma H_k\dd A
 =
 \frac{\sphereareanm}{n}\int_0^L
 \left(
   k\dot\theta\sin^{k-1}\theta\,x^{n-k}
   +(n-k)\sin^k\theta\,x^{n-k-1}
 \right)\dd s .
\end{equation}
Since $\frac{\dd}{\dd s}(J_k\circ\theta)
=\dot\theta\sin^{k-1}\theta$ almost everywhere and
$n-k\ge1$, integration by parts and $x(0)=x(L)=0$ yield
\begin{align*}
 \int_0^L\dot\theta\sin^{k-1}\theta\,x^{n-k}\dd s
 &=
 \bigl[J_k(\theta)x^{n-k}\bigr]_0^L-(n-k)\int_0^L
 J_k(\theta)\cos\theta\,x^{n-k-1}\dd s\\
 &=
 -(n-k)\int_0^L
 J_k(\theta)\cos\theta\,x^{n-k-1}\dd s .
\end{align*}
Substitution in \eqref{eq:H-formula} gives
\begin{equation}\label{eq:body}
 \int_\Sigma H_k\dd A
 =
 \frac{\sphereareanm(n-k)}{n}
 \int_0^L F_k(\theta(s))x(s)^{n-k-1}\dd s .
\end{equation}
By \cref{lem:weighted-coordinate-calculus}, the change of variables
$\dd\tau=x(s)^{m-1}\dd s$ transforms \eqref{eq:body}
into \eqref{eq:body-tau}.
\end{proof}

Combining \eqref{eq:area-element-prelim} with \eqref{eq:body}, and
using $F_0\equiv1$, gives the unified formula
\begin{equation}\label{eq:unified-Mj-profile}
 \mathcal M_j(\Sigma)
 =\frac{\sphereareanm(n-j)}{n}
   \int_0^L F_j(\theta(s))x(s)^{n-j-1}\dd s,
 \qquad 0\le j\le n-1.
\end{equation}
For $j=0$ this is simply the area formula; for $j\ge1$ it is the
integration-by-parts identity above.
Moreover, $F_j$ is strictly positive on $(0,\pi]$ for $j\ge1$.
Admissibility forces the angle to be positive on a set of positive
measure, so $\mathcal M_j(\Sigma)>0$ for every $0\le j\le n-1$; all
normalized powers used below are therefore well defined.

\section{Rearrangement and convex comparison}
\label{sec:rearrangement}

\subsection{Construction of the convex comparison body}
\label{subsec:comparison-body-construction}

We retain the notation of \cref{sec:weighted-form}, with $k$ fixed
and $m=n-k$.  By \cref{prop:monotone-rearrangements}, the nondecreasing
rearrangement of $\vartheta$ in the weighted coordinate $\tau$ is
\begin{equation}
	\vartheta^\uparrow(\tau)
	=\sup\Bigl\{a\in[0,\pi]:
	\bigl|\{r\in(0,T):\vartheta(r)\ge a\}\bigr|
	\ge T-\tau\Bigr\},
	\qquad 0\le\tau\le T.
	\label{eq:weighted-angle-rearrangement-explicit}
\end{equation}
Since cosine is decreasing on $[0,\pi]$,
\begin{equation}
	(\cos\vartheta)^\downarrow(\tau)=\cos\vartheta^\uparrow(\tau).
	\label{eq:weighted-slope-rearrangement-explicit}
\end{equation}

We first show that the rearranged angle generates a convex comparison
body $K_k$ with $C^{1,1}$ boundary.  The following result reconstructs a meridian from the rearranged angle by restoring the arclength parametrization, while ensuring a positive interior radius and the required endpoint conditions.

\begin{proposition}
\label{prop:rearranged-meridian}
Let $\vartheta^\uparrow$ be the nondecreasing rearrangement of $\vartheta$, and define
\[
  Y^*(\tau)=\int_0^\tau\cos\vartheta^\uparrow(r)\dd r,
  \qquad
  X^*=(mY^*)^{1/m}.
\]
Introduce the new parameter by
\[
 S(\tau)=\int_0^\tau(X^*(r))^{1-m}\dd r,
 \qquad L^*=S(T),
\]
let $\lambda:[0,L^*]\to[0,T]$ be the inverse of $S$, and, for
$0\le\sigma\le L^*$, set
\[
 \widetilde x(\sigma)=X^*(\lambda(\sigma)),
 \qquad
 \widetilde z(\sigma)=\int_0^\sigma
 \sin\widetilde\vartheta(r)\dd r,
 \qquad
 \widetilde\vartheta(\sigma)=\vartheta^\uparrow(\lambda(\sigma)).
\]
Then 
$
 \widetilde x,\widetilde z\in C^1([0,L^*])
$
and
$
 \widetilde\vartheta\in C([0,L^*])
$
with $\widetilde\vartheta$ nondecreasing.
The reconstructed meridian
$\widetilde\gamma=(\widetilde x,\widetilde z)$ has unit speed,
\[
 \widetilde x'=\cos\widetilde\vartheta,
 \qquad
 \widetilde z'=\sin\widetilde\vartheta
 \quad(0\le\sigma\le L^*).
\]
Moreover,
\[
 \widetilde x(0)=\widetilde x(L^*)=0,
 \qquad
 \widetilde x>0\ \text{on }(0,L^*),
\]
\[
 \widetilde\vartheta(0)=0,
 \qquad
 \widetilde\vartheta(L^*)=\pi,
 \qquad
 \widetilde z(L^*)>0.
\]
\end{proposition}

\begin{proof}
Applying \eqref{eq:HL} to $\cos\vartheta$ gives
\[
 Y^*\ge Y\quad\text{on }[0,T],
 \qquad
Y^*(T)= Y(T)=0,
\]
and therefore
\begin{equation}\label{eq:restruct-radial-function}
	 X^*\ge X>0\quad\text{on }(0,T),
	\qquad X^*(0)=X^*(T)=0.
\end{equation}
Because $m\ge1$, the above radial comparison and
\cref{lem:weighted-coordinate-calculus} give
\[
 0<L^*=\int_0^T(X^*)^{1-m}\dd\tau
 \le\int_0^T X^{1-m}\dd\tau=L.
\]
Thus $S$ is well defined and strictly increasing, with
\[
 S'=(X^*)^{1-m}>0\quad(0<\tau<T).
\]
The inverse $\lambda=S^{-1}$ is continuous on $[0,L^*]$ and $C^1$
in the interior, where
\[
 \lambda'(\sigma)
 =\frac{1}{S'(\lambda(\sigma))}
 =\widetilde x(\sigma)^{m-1}.
\]
For $m>1$, the right-hand side extends continuously to zero at both
endpoints.  When $m=1$, one has
$S(\tau)=\tau$, $\lambda(\sigma)=\sigma$, and $L^*=T=L$.
The mean value theorem at the endpoints therefore yields
\[
 \lambda\in C^1([0,L^*]),
 \qquad
 \lambda'=\widetilde x^{m-1}
 \quad(0\le\sigma\le L^*).
\]

The angle $\widetilde\vartheta=\vartheta^\uparrow\circ\lambda$ is
continuous and nondecreasing.  Hence $\widetilde z\in C^1([0,L^*])$
and $\widetilde z'=\sin\widetilde\vartheta$.  On $(0,L^*)$, one has
\[
 \widetilde x'=(X^*)'(\lambda)\lambda'
 =\widetilde x^{1-m}\widetilde x^{m-1}\cos\widetilde\vartheta
 =\cos\widetilde\vartheta .
\]
Continuity of the right-hand side and the mean value theorem
give $\widetilde x\in C^1([0,L^*])$.  Thus
\[
 |\widetilde\gamma'|^2
 =\cos^2\widetilde\vartheta+\sin^2\widetilde\vartheta=1
 \quad\text{on }[0,L^*],
\]
so $\sigma$ is the arclength of the reconstructed meridian $\widetilde{\gamma}$.
Since $\lambda(0)=0$ and $\lambda(L^*)=T$,
\[
\widetilde x(0)=X^*(0)=0,\quad \widetilde x(L^*)=X^*(T)=0,\quad \widetilde\vartheta(0)=\vartheta^\uparrow(0)=0,\quad \widetilde\vartheta(L^*)=\vartheta^\uparrow(T)=\pi.
\]
By \eqref{eq:restruct-radial-function}, $\widetilde x>0$ on $(0,L^*)$.
Finally,
\[
 \widetilde z(L^*)
 =\int_0^{L^*}\sin\widetilde\vartheta(\sigma)\dd\sigma>0.
\]
Indeed, $\widetilde\vartheta$ takes values in $[0,\pi]$ and, by
continuity, lies in $(0,\pi)$ on a nonempty
open interval.
\end{proof}

The preceding construction yields a continuous monotone tangent angle $\widetilde\vartheta$.  The original angle $\theta$ is Lipschitz in arclength $s$, whereas $\vartheta(\tau)=\theta(s(\tau))$ may not be Lipschitz when $m=n-k>1$, because $s'$ may blow up at the two poles.  The following lemma restores Lipschitz continuity after reparametrization by the arclength of the reconstructed meridian. The argument may be viewed as a one-dimensional weighted
analogue of the P\'olya--Szeg\H{o} principle; compare
\cite[Theorem~4.2]{NobiliViolo}.

\begin{lemma}
\label{lem:weighted-Lipschitz-preservation}
With the notation of \cref{prop:rearranged-meridian}, the rearranged
angle in the reconstructed arclength satisfies the following estimate.
Write
$\Lip(\theta):=\|\dot\theta\|_{L^\infty(0,L)}$.  Then, for every
$0\le\sigma_1<\sigma_2\le L^*$,
\begin{equation}\label{eq:weighted-Lipschitz-preservation}
 0\le
 \widetilde\vartheta(\sigma_2)-\widetilde\vartheta(\sigma_1)
 \le \Lip(\theta)(\sigma_2-\sigma_1).
\end{equation}
In particular,
$\widetilde\vartheta\in W^{1,\infty}(0,L^*)$ and
$\|\widetilde\vartheta'\|_{L^\infty(0,L^*)}
 \le \Lip(\theta)$.
\end{lemma}

\begin{proof}
Set
\[
 \dd\mu=x^{m-1}\dd s,
 \qquad
 \mu(E):=\int_E x(s)^{m-1}\dd s.
\]
Since $\theta(L)-\theta(0)=\pi$, one has
$\Lip(\theta)\ge\pi/L>0$.
For $0<t<\pi$, define the right-continuous functions
\[
 D(t):=\mu(\{\theta\le t\}),
 \qquad
 Q(t):=S(D(t)).
\]
Equimeasurability and the change of variables $\sigma=S(\tau)$ then give
\begin{equation}\label{eq:arclength-angle-distribution}
 Q(t)=|\{\widetilde\vartheta\le t\}|.
\end{equation}
We shall prove the lower growth bound
\[
 Q(b)-Q(a)\ge\frac{b-a}{\Lip(\theta)}
 \qquad(0<a<b<\pi).
\]

For every \(t\in(0,\pi)\), the endpoint conditions and continuity imply \(0<D(t)<T\). Since $D$ is nondecreasing, the differentiability theorem for monotone functions gives that $D'$ exists for almost every $t\in(0,\pi)$. By \cref{prop:rearranged-meridian}, $S\in C^1((0,T))$ and $S'=(X^*)^{1-m}$. Therefore
\begin{equation*}
Q'(t)=S'(D(t))D'(t)={X^*(D(t))}^{1-m}D'(t)
\quad\text{for a.e. }t\in(0,\pi).
\end{equation*}
Define
\[
\mathcal R
:=
\{s:\theta \text{ is differentiable at }s
\text{ and }\dot\theta(s)\neq 0\},
\]
\[
\mathcal C
:=
\{s:\theta \text{ is differentiable at }s
\text{ and }\dot\theta(s)=0\}.
\]
By Rademacher's theorem, 
\[
D(t)=D_{\mathcal R}(t)+D_{\mathcal C}(t),
\]
where
\[
D_{\mathcal R}(t)
:=\mu(\mathcal R\cap\{\theta\le t\}),
\qquad
D_{\mathcal C}(t)
:=\mu(\mathcal C\cap\{\theta\le t\}).
\]
For any $a<b$,
\[
	D_{\mathcal R}(b)-D_{\mathcal R}(a)
	=
	\int_{\mathcal R\cap\{a<\theta\leq b\}}x(s)^{m-1}\,ds
	=
	\int_a^b
	\sum_{s\in\mathcal R\cap\theta^{-1}(t)}
	\frac{x(s)^{m-1}}{|\dot\theta(s)|}\,dt,
\]
where the second identity comes from the one-dimensional coarea formula. 
Consequently,
\begin{equation}\label{eq:weighted-distribution-regular-derivative}
	D_{\mathcal R}'(t)
	=
	\sum_{s\in\mathcal R\cap\theta^{-1}(t)}
	\frac{x(s)^{m-1}}{|\dot\theta(s)|}
	\qquad\text{for a.e. }t.
\end{equation}
On the other hand, since $\dot\theta=0$ on $\mathcal C$, another application of the
one-dimensional coarea formula yields
\[
\int_{\mathbb R}
\sum_{s\in\mathcal C\cap\theta^{-1}(t)}1\,dt
=
\int_{\mathcal C}|\dot\theta(s)|\,ds
=
0.
\]
It follows that $|\theta(\mathcal C)|=0$.
Therefore $D_{\mathcal C}$ is a nondecreasing function whose
Stieltjes measure is singular with respect to Lebesgue measure.
By the Lebesgue decomposition theorem for monotone functions,
\[
D_{\mathcal C}'(t)=0
\qquad\text{for a.e. }t.
\]
It follows that
\begin{equation}\label{eq:weighted-distribution-derivative}
	D'(t)
	=
	D_{\mathcal R}'(t)+D_{\mathcal C}'(t)
	=
	D_{\mathcal R}'(t)
	\qquad\text{for a.e. }t.
\end{equation}
The level set of $\theta$ is not necessarily finite for every
$t$. The claim is only that it is finite for almost every $t$.
Define the multiplicity function
\[
N(t):=
\#\bigl\{s\in[0,L]:\theta(s)=t\bigr\}
=
\#\theta^{-1}(t).
\]
Since $\theta\colon[0,L]\to\mathbb R$ is Lipschitz, the
one-dimensional area formula gives
\[
\int_{\mathbb R}N(t)\,dt
=
\int_0^L|\dot\theta(s)|\,ds
\leq
L\,\operatorname{Lip}(\theta)
<
\infty.
\]
Moreover, since
$
\theta(0)=0
$
and
$
\theta(L)=\pi,
$
the continuity of $\theta$ implies
that
\[
1\leq N(t)<\infty
\qquad\text{for a.e. }t\in(0,\pi).
\]
Let $\mathcal N_\theta$ be the set where $\theta$ is not
differentiable.  Since $\theta$ is Lipschitz, it maps Lebesgue-null
sets to Lebesgue-null sets, so $|\theta(\mathcal N_\theta)|=0$.
Together with $|\theta(\mathcal C)|=0$ and the almost-everywhere
finiteness of the multiplicity, this allows us to choose, after also excluding the null sets where \eqref{eq:weighted-distribution-regular-derivative} and \eqref{eq:weighted-distribution-derivative} fail, a full-measure set
$\mathcal G\subset(0,\pi)$ such that, for every $t\in\mathcal G$, the
level set is finite and regular,
\[
 \theta^{-1}(t)=\{s_1,\ldots,s_{N(t)}\},
 \qquad
 \dot\theta(s_i)\ne0\quad(1\le i\le N(t)),
\]
and $D$ is differentiable at $t$. By \eqref{eq:weighted-distribution-regular-derivative} and \eqref{eq:weighted-distribution-derivative},
\begin{equation}\label{eq:weighted-distribution-derivative-g}
 D'(t)
 =\sum_{i=1}^{N(t)}\frac{x(s_i)^{m-1}}{|\dot\theta(s_i)|}
 \ge\frac1{\Lip(\theta)}\sum_{i=1}^{N(t)} x(s_i)^{m-1}
 \qquad\text{for }t\in\mathcal G.
\end{equation}
Set $R(t):=X^*(D(t))$, the radius at the weighted coordinate $D(t)$.
Since $D(t)\in(0,T)$, one has $R(t)>0$. For every $t\in\mathcal G$, the level set $\theta^{-1}(t)$ is finite and hence $\mu(\{\theta=t\})=0$. In particular, $D(t)=\mu(\{\theta<t\})$. Write the finitely many relative
components of $\{\theta<t\}$ in $[0,L]$ as $I_j$, and denote their
endpoints by $\alpha_j<\beta_j$, allowing $\alpha_j=0$ or
$\beta_j=L$.  Applying 
\eqref{eq:rearrangement-functional-invariance} with 
$\Psi_t(u)=\cos u\,\mathbf 1_{\{u<t\}}$ gives
\[
 \int_0^{D(t)}\cos\vartheta^\uparrow(\tau)\dd\tau
 =\int_{\{\theta<t\}}\cos\theta\dd\mu.
\]
Combining this identity with $(x^m)'=m x^{m-1}\cos\theta$, we obtain
\begin{align}
 R(t)^m
 &=m\int_0^{D(t)}\cos\vartheta^\uparrow(\tau)\dd\tau
 \notag\\
 &=m\int_{\{\theta<t\}}x^{m-1}\cos\theta\dd s
 =\int_{\{\theta<t\}}(x^m)'\dd s
 \notag\\
 &=\sum_j\bigl(x(\beta_j)^m-x(\alpha_j)^m\bigr)
 \le\sum_{i=1}^{N(t)} x(s_i)^m .
 \label{eq:weighted-level-primitive}
\end{align}
Indeed, after discarding the nonpositive terms
$-x(\alpha_j)^m$, every remaining positive right-endpoint term comes
from a distinct $s_i$; the possible endpoint contributions at $0$ and
$L$ vanish because $x(0)=x(L)=0$.  If $m>1$, monotonicity of
finite-dimensional $\ell^p$ norms gives
\[
 R(t)^{m-1}
 \le\left(\sum_{i=1}^{N(t)} x(s_i)^m\right)^{(m-1)/m}
 \le\sum_{i=1}^{N(t)} x(s_i)^{m-1}.
\]
For $m=1$, the desired inequality reads $1\le N(t)$.
Thus, in every case,
\begin{equation}\label{eq:weighted-level-radius-bound}
 R(t)^{m-1}\le\sum_{i=1}^{N(t)} x(s_i)^{m-1}.
\end{equation}
Combining
\eqref{eq:weighted-distribution-derivative-g} and
\eqref{eq:weighted-level-radius-bound} gives
\[
 Q'(t)=R(t)^{1-m}D'(t)\ge\frac1{\Lip(\theta)}
 \quad\text{for a.e. }t\in(0,\pi).
\]

Since $Q$ is nondecreasing and right-continuous, its Lebesgue--Stieltjes
measure satisfies
\[
 \mathrm{d}Q\ge Q'(t)\,\mathrm{d}t
 \ge\frac1{\Lip(\theta)}\,\mathrm{d}t.
\]
Consequently,
\begin{equation}\label{eq:inverse-angle-global-bound}
 Q(b)-Q(a)=\mathrm{d}Q((a,b])
 \ge\frac{b-a}{\Lip(\theta)}
 \qquad(0<a<b<\pi).
\end{equation}
Now let $0\le\sigma_1<\sigma_2\le L^*$ and put
$a=\widetilde\vartheta(\sigma_1)$ and
$b=\widetilde\vartheta(\sigma_2)$.  If $a<b$, choose
$a<a'<b'<b$.  By \eqref{eq:arclength-angle-distribution} and
monotonicity,
\[
 [0,\sigma_1]\subset\{\widetilde\vartheta\le a'\},
 \qquad
 \{\widetilde\vartheta\le b'\}\subset[0,\sigma_2),
\]
and hence
\[
 \sigma_1\le Q(a'),
 \qquad
 Q(b')\le\sigma_2.
\]
Thus \eqref{eq:inverse-angle-global-bound} gives
\[
 \frac{b'-a'}{\Lip(\theta)}
 \le Q(b')-Q(a')
 \le\sigma_2-\sigma_1.
\]
Letting
$a'\downarrow a$ and $b'\uparrow b$ proves
\[
 \widetilde\vartheta(\sigma_2)
 -\widetilde\vartheta(\sigma_1)
 \le \Lip(\theta)(\sigma_2-\sigma_1);
\]
the same argument covers $a=0$ or $b=\pi$, and the case $a=b$ is
immediate.  Together with the monotonicity of
$\widetilde\vartheta$, this proves
\eqref{eq:weighted-Lipschitz-preservation}.
\end{proof}

Combining the reconstruction with the preceding Lipschitz estimate now
completes the construction of the comparison body.

\begin{corollary}
\label{cor:rearranged-C11-closure}
With the notation of \cref{prop:rearranged-meridian}, the
arclength-reconstructed meridian $\widetilde\gamma$ belongs to
$C^{1,1}([0,L^*];\R^2)$ and generates the boundary
$\partial K_k$ of a convex
axisymmetric body with $C^{1,1}$ boundary.
\end{corollary}

\begin{proof}
	By \cref{prop:rearranged-meridian}, \cref{lem:weighted-Lipschitz-preservation} and \cite[Proposition~3.1]{DHMT}, $\widetilde\gamma$ is an injective
	unit-speed $C^{1,1}$ meridian. Near each pole, $\widetilde{\vartheta}(0)=0$ and $\widetilde{\vartheta}(L^*)=\pi$ give $\widetilde z=\phi(\widetilde x)$ with
	$\phi\in C^{1,1}$ and $\phi'(0)=0$, so the radial extension
	$y\mapsto\phi(|y|)$ is $C^{1,1}$.
	Thus its revolution is a closed embedded $C^{1,1}$ hypersurface
	enclosing a compact connected region $K_k$ with nonempty interior.
			
			Since $\widetilde\vartheta$ is nondecreasing and takes values
			in $[0,\pi]$,
			\[
			\widetilde\kappa_1=\widetilde\vartheta'\ge0,\qquad
			\widetilde\kappa_2=\cdots=\widetilde\kappa_n
			=\frac{\sin\widetilde\vartheta}{\widetilde x}\ge0
			\quad\text{a.e.}
			\]
			These curvature bounds imply that $K_k$ is locally convex.
			Indeed, near each boundary point, write $\partial K_k$ as the
			graph of a $C^{1,1}$ function $u$, with $K_k$ lying above the graph. With
			the outward normal convention, the nonnegativity of the second fundamental form gives $D^2u\ge0$ almost everywhere.  Since $Du$ is Lipschitz, there is no
			additional singular second-derivative measure, so the graph is
			locally convex.  This gives local convexity of the
			enclosed set $K_k$.
			Since $K_k$ is compact and connected, the Tietze--Nakajima
			theorem \cite[Theorem~1.3]{ManiLevitska1993} yields its convexity.
\end{proof}

\subsection{Primitive rearrangement and curvature comparison}
\label{subsec:primitive-curvature-comparison}

The geometric construction is now complete: the rearranged angle
generates a convex body with $C^{1,1}$
boundary.  To compare its curvature functionals with those of the
original hypersurface, we first prove the following 
primitive rearrangement inequality.

For $0\le j\le n-1$ define
\begin{equation}\label{eq:Aj-definition}
 A_j(c):=F_j(\arccos c),\qquad -1\le c\le1,
\end{equation}
where $F_0\equiv1$.

\begin{proposition}
\label{prop:primitive-rearrangement}
Let $T>0$, and for $f\in L^1(0,T)$ write
\[
 (\mathcal P f)(t):=\int_0^t f(r)\dd r.
\]
Let $c\in L^\infty(0,T)$ take values in $[-1,1]$,
and assume that $\mathcal P c\ge0$ on $[0,T]$.  Let $c^\downarrow$ be
the decreasing rearrangement from
\cref{prop:monotone-rearrangements}.  By
\eqref{eq:HL},
\[
 \mathcal P(c^\downarrow)\ge \mathcal P c\ge0
 \qquad\text{on }[0,T].
\]
Then, for $0\le j\le n-1$ and $a>0$,
\begin{equation}\label{eq:primitive-rearrangement}
 \int_0^T A_j(c(t))[(\mathcal P c)(t)]^a\dd t
 \le
 \int_0^T A_j(c^\downarrow(t))
       [(\mathcal P(c^\downarrow))(t)]^a\dd t.
\end{equation}
If, in addition, $c$ is continuous on $[0,T]$ and
$\mathcal P c>0$ on $(0,T)$, equality in
\eqref{eq:primitive-rearrangement} holds if and only if $c$ is nonincreasing on $[0,T]$.
\end{proposition}

\begin{proof}
Fix $j$ and $a$, and write
\[
 \mathscr E(c):=\int_0^T A_j(c(t))[(\mathcal P c)(t)]^a\dd t.
\]
We prove the
inequality \eqref{eq:primitive-rearrangement} by approximating $c$ with step functions.  For
$\nu\in\mathbb N$, put
\[
 h_\nu=\frac{T}{2^\nu},
 \qquad
 I_{\nu,i}=[ih_\nu,(i+1)h_\nu),
 \qquad 0\le i\le2^\nu-1,
\]
with the last interval understood to contain $T$, and define
\[
 c_\nu(t)=\frac1{h_\nu}\int_{I_{\nu,i}}c(r)\dd r
 \quad\text{for }t\in I_{\nu,i}.
\]
Then $-1\le c_\nu\le1$, and 
$
 (\mathcal P c_\nu)(ih_\nu)=(\mathcal P c)(ih_\nu)\ge0
$
for
$
0\le i\le2^\nu.
$
Since $\mathcal P c_\nu$ is affine on each block, it is nonnegative on
$[0,T]$.

Suppose that two adjacent blocks of the step function
have values $p<q$.  Use a local coordinate with origin at the left
endpoint of their union, and let $\rho\ge0$ be the value of the primitive there.  Define on $(0,2h_\nu)$
\[
 \begin{aligned}
 c_{pq}&=p\mathbf 1_{(0,h_\nu)}
          +q\mathbf 1_{(h_\nu,2h_\nu)},\\
 c_{qp}&=q\mathbf 1_{(0,h_\nu)}
          +p\mathbf 1_{(h_\nu,2h_\nu)},\\
 u_{pq}(t)&=\rho+\int_0^t c_{pq}(r)\dd r,\\
 u_{qp}(t)&=\rho+\int_0^t c_{qp}(r)\dd r.
 \end{aligned}
\]
Since $u_{pq}$ is the primitive on this pair, $u_{pq}\ge0$.
Direct integration gives
\[
 u_{qp}(t)-u_{pq}(t)
 =(q-p)\min\{t,h_\nu,2h_\nu-t\}\ge0
 \qquad(0\le t\le2h_\nu).
\]
In particular, $u_{qp}\ge u_{pq}\ge0$ and $u_{qp}(2h_\nu)=u_{pq}(2h_\nu)$.
It follows from a direct calculation that
\begin{equation}\label{eq:dyadic-exchange-gain}
	\begin{aligned}
		&\int_0^{2h_\nu}A_j(c_{qp})u_{qp}^a\dd t
		-\int_0^{2h_\nu}A_j(c_{pq})u_{pq}^a\dd t\\
		=&
		A_j(q)\int_0^{h_\nu}\Bigl[(\rho+qt)^a-(\rho+ph_\nu+qt)^a\Bigr]\dd t\\
		&-A_j(p)\int_0^{h_\nu}\Bigl[(\rho+pt)^a-(\rho+qh_\nu+pt)^a\Bigr]\dd t.
	\end{aligned}
\end{equation}
Set
\[
 \Phi_a(r)=\frac{r^{a+1}}{a+1},\qquad r\ge0.
\]
For $pq\ne0$, direct integration gives
\begin{equation*}
	\begin{aligned}
		&\int_0^{2h_\nu}A_j(c_{qp})u_{qp}^a\dd t
		-\int_0^{2h_\nu}A_j(c_{pq})u_{pq}^a\dd t\\
		=&
		\left(\frac{A_j(q)}q-\frac{A_j(p)}p\right)
		\Bigl[
		\Phi_a(\rho+ph_\nu)+\Phi_a(\rho+qh_\nu)-\Phi_a(\rho)
		-\Phi_a(\rho+(p+q)h_\nu)
		\Bigr]\\
		=&
		\left(\frac{A_j(q)}q-\frac{A_j(p)}p\right)
		ph_\nu\int_0^1
		\Bigl[\Phi_a'(\rho+ph_\nu t)
		-\Phi_a'(\rho+qh_\nu+ph_\nu t)\Bigr]\dd t.
	\end{aligned}
\end{equation*}
The four arguments of $\Phi_a$ are nonnegative, being $\rho$,
$u_{pq}(h_\nu)$, $u_{qp}(h_\nu)$, and $u_{pq}(2h_\nu)$.
The derivative $\Phi_a'(r)=r^a$ is continuous and strictly increasing on
$[0,\infty)$.  Therefore,
\[
 ph_\nu\int_0^1
 \Bigl[\Phi_a'(\rho+ph_\nu t)
       -\Phi_a'(\rho+qh_\nu+ph_\nu t)\Bigr]\dd t
\]
has sign $-\operatorname{sgn}(pq)$.  Moreover,
\[
 \frac{A_j(q)}q-\frac{A_j(p)}p
 =-\frac{\Delta_j(p,q)}{pq},
 \qquad
 \Delta_j(p,q):=qA_j(p)-pA_j(q).
\]
For $j\ge1$, differentiation of
$A_j(c)=F_j(\arccos c)$ and \eqref{eq:Fj-monotone} gives
\[
 A_j'(c)=-jJ_j(\arccos c),\qquad -1<c<1.
\]
Consequently, for $c\in(-1,0)\cup(0,1)$,
\begin{equation}\label{eq:Aj-over-c-derivative}
 \frac{\dd}{\dd c}\left(\frac{A_j(c)}{c}\right)
 =\frac{cA_j'(c)-A_j(c)}{c^2}
 =-\frac{(1-c^2)^{j/2}}{c^2}<0.
\end{equation}
Hence $A_j(c)/c$ is strictly decreasing separately on
 $(-1,0)$ and $(0,1)$. This proves that $\Delta_j(p,q)\ge0$ when $pq>0$.  When $p<0<q$, it follows directly from
$A_j\ge0$.  For $j=0$, it is simply
$\Delta_0(p,q)=q-p>0$.
Thus, for $pq\ne0$, 
\begin{equation*}
	\int_0^{2h_\nu}A_j(c_{pq})u_{pq}^a\dd t\le\int_0^{2h_\nu}A_j(c_{qp})u_{qp}^a\dd t.
\end{equation*}
If $p=0$ or $q=0$, the same inequality follows directly from \eqref{eq:dyadic-exchange-gain}.

Repeating these exchanges places the
finitely many values of $c_\nu$ in nonincreasing order.  Since the
blocks have equal length, the resulting function agrees almost
everywhere with $c_\nu^\downarrow$.  Therefore
\begin{equation}\label{eq:comparison-step-function}
	\mathscr E(c_\nu)\le \mathscr E(c_\nu^\downarrow).
\end{equation}
By the Lebesgue differentiation theorem,
\[
	 c_\nu(t)\longrightarrow c(t)
	 \qquad\text{for a.e. }t\in(0,T).
\]
Since $|c_\nu-c|\le2$, dominated convergence gives
\[
	\|c_\nu-c\|_{L^1(0,T)}\longrightarrow0.
\]
Moreover,
\[
	\|\mathcal P c_\nu-\mathcal P c\|_{L^\infty(0,T)}
	\le \|c_\nu-c\|_{L^1(0,T)}
	\longrightarrow0.
\]
Thus
\[
\begin{aligned}
	|\mathscr E(c_\nu)-\mathscr E(c)|
	={}&\left|\int_0^T
	\bigl[A_j(c_\nu)-A_j(c)\bigr](\mathcal Pc_\nu)^a\,\dd t
	+\int_0^T
	A_j(c)\bigl((\mathcal Pc_\nu)^a-(\mathcal Pc)^a\bigr)\,\dd t\right|\\
	\le{}&
	T^a\|A_j(c_\nu)-A_j(c)\|_{L^1(0,T)}\\
	&+T\|A_j\|_{L^\infty([-1,1])}
	\|(\mathcal Pc_\nu)^a-(\mathcal Pc)^a\|_{L^\infty(0,T)}.
\end{aligned}
\]
The first term tends to $0$ by dominated convergence, since $A_j$ is
continuous on $[-1,1]$ and $c_\nu\to c$ almost everywhere.
The second tends to $0$ because $\mathcal P c_\nu\to\mathcal P c$
uniformly and $r\mapsto r^a$ is uniformly continuous on $[0,T]$.
Therefore,
\[
\mathscr E(c_\nu)\longrightarrow\mathscr E(c).
\]
For the rearranged functions, the contraction estimate
\eqref{eq:rearrangement-L1-contraction} gives
\[
	\|c_\nu^\downarrow-c^\downarrow\|_{L^1(0,T)}
	\le \|c_\nu-c\|_{L^1(0,T)}\longrightarrow0,
\]
and hence
\[
	\|\mathcal P(c_\nu^\downarrow)-\mathcal P(c^\downarrow)\|_{L^\infty(0,T)}
	\le \|c_\nu^\downarrow-c^\downarrow\|_{L^1(0,T)}
	\longrightarrow0.
\]
The uniform continuity of $A_j$ on $[-1,1]$ also gives
\[
 \|A_j(c_\nu^\downarrow)-A_j(c^\downarrow)\|_{L^1(0,T)}
 \longrightarrow0.
\]
The preceding energy estimate therefore yields
\[
\mathscr E(c_\nu^\downarrow)\longrightarrow
\mathscr E(c^\downarrow).
\]
Passing to the limit in \eqref{eq:comparison-step-function} proves
\eqref{eq:primitive-rearrangement}.

Suppose that equality in \eqref{eq:primitive-rearrangement} holds.
If $c$ is nonincreasing, it
agrees almost everywhere with $c^\downarrow$ and equality is immediate.
Conversely, suppose that $c$ is not nonincreasing.  By continuity there
exist $0<r<s<T$ such that
\[
 p:=c(r)<c(s)=:q.
\]
By continuity, choose $h_0>0$ sufficiently small that
$h_0<\min\{s-r,T-s\}$ and
\[
 c(s+v)>c(r+v)\qquad(0\le v\le h_0).
\]
For $0<h<h_0$, define
\[
 c_h(t):=
 \begin{cases}
  c\bigl(t+(s-r)\bigr),&r<t<r+h,\\
  c\bigl(t-(s-r)\bigr),&s<t<s+h,\\
  c(t),&\text{otherwise}.
 \end{cases}
\]
Direct integration gives
\[
 (\mathcal P c_h)(t)-(\mathcal P c)(t)=
 \begin{cases}
  \displaystyle\int_0^{t-r}\bigl(c(s+v)-c(r+v)\bigr)\dd v,
      &r\le t< r+h,\\[2mm]
  \displaystyle\int_0^h\bigl(c(s+v)-c(r+v)\bigr)\dd v,&r+h\le t\le s,\\[1mm]
  \displaystyle\int_{t-s}^h\bigl(c(s+v)-c(r+v)\bigr)\dd v,
      &s< t\le s+h,\\[2mm]
  0,&t\notin[r,s+h].
 \end{cases}
\]
Set
\[
\delta_h
:=
\int_0^h\bigl(c(s+v)-c(r+v)\bigr)\dd v.
\]
Therefore, $\mathcal P c_h\ge\mathcal P c\ge0$ and 
$
0\le \mathcal P c_h-\mathcal P c
\le\delta_h\le2h.
$
In particular,
\[
\|\mathcal P c_h-\mathcal P c\|_{L^\infty(0,T)}
\longrightarrow0.
\]
Since $\mathcal Pc>0$ on $(0,T)$ and $-1\le c_h\le1$, for
$0<h<h_0$ and $t\in[r,s+h]$ we have
\[
0<\min_{[r,s+h_0]}\mathcal Pc
\le (\mathcal Pc)(t)\le (\mathcal Pc_h)(t)\le T.
\]

We next compute the first variation of the functional $\mathscr E(c)$.
Write
\[
 \Delta_h(t):=
 A_j(c_h(t))\bigl((\mathcal P c_h)(t)\bigr)^a
 -A_j(c(t))\bigl((\mathcal P c)(t)\bigr)^a.
\]
Since \(c_h=c\) and \(\mathcal P c_h=\mathcal P c\) for $t\notin[r,s+h]$, one has
\[
 \frac{\mathscr E(c_h)-\mathscr E(c)}h
 =\frac1h\int_r^{r+h}\Delta_h\dd t
 +\frac1h\int_{r+h}^{s}\Delta_h\dd t
 +\frac1h\int_s^{s+h}\Delta_h\dd t.
\]
Recall that $\mathcal Pc$ has the uniform positive lower bound
$\min_{[r,s+h_0]}\mathcal Pc>0$ on $[r,s+h]$.
On \([r,r+h]\), one has $c_h(t)=c(t+s-r)$. Since
$\|\mathcal P c_h-\mathcal P c\|_{L^\infty(0,T)}\to0$ and $A_j\circ c$
is continuous, the integrand converges uniformly, and
\[
 \frac1h\int_r^{r+h}\Delta_h\dd t
 \longrightarrow
 \bigl(A_j(q)-A_j(p)\bigr)\bigl((\mathcal Pc)(r)\bigr)^a.
\]
On \([r+h,s]\), one has $c_h=c$ and
$\mathcal P c_h=\mathcal P c+\delta_h$ with $\delta_h/h\to q-p$.
Uniform differentiability of $x\mapsto x^a$ on the compact interval
$[\min_{[r,s+h_0]}\mathcal Pc,T]\subset(0,\infty)$ then gives
\[
\begin{aligned}
	 \frac1h\int_{r+h}^{s}\Delta_h\dd t
	&=\frac{\delta_h}{h}
	\int_{r+h}^{s}
	A_j(c(t))
	\frac{\bigl((\mathcal P c)(t)+\delta_h\bigr)^a
		-\bigl((\mathcal P c)(t)\bigr)^a}{\delta_h}\dd t\\
	&\longrightarrow
	a(q-p)\int_r^s
	A_j(c(t))\bigl((\mathcal P c)(t)\bigr)^{a-1}\dd t.
\end{aligned}
\]
On \([s,s+h]\), $c_h(t)=c(t-s+r)\to p$ and $c(t)\to q$ uniformly, while
both primitives converge uniformly to $(\mathcal Pc)(s)$. Hence
\[
 \frac1h\int_s^{s+h}\Delta_h\dd t
 \longrightarrow
 \bigl(A_j(p)-A_j(q)\bigr)\bigl((\mathcal Pc)(s)\bigr)^a.
\]
Combining the three limits yields
\[
\begin{aligned}
 \lim_{h\rightarrow0}\frac{\mathscr E(c_h)-\mathscr E(c)}h
 ={}&\bigl(A_j(q)-A_j(p)\bigr)
 \left[
 \bigl((\mathcal P c)(r)\bigr)^a
 -\bigl((\mathcal P c)(s)\bigr)^a
 \right]\\
 &+a(q-p)\int_r^s
 A_j(c(t))
 \bigl((\mathcal P c)(t)\bigr)^{a-1}
 \dd t.
\end{aligned}
\]
Since
\[
\bigl((\mathcal P c)(s)\bigr)^a
-\bigl((\mathcal P c)(r)\bigr)^a
=
a\int_r^s
c(t)\bigl((\mathcal P c)(t)\bigr)^{a-1}
\dd t,
\]
we conclude that
\begin{equation}\label{eq:strict-fragment-first-variation}
 \lim_{h\rightarrow0}\frac{\mathscr E(c_h)-\mathscr E(c)}h
 =a(q-p)\int_r^s\bigl(A_j(c(t))-\lambda_{p,q}c(t)\bigr)
       \bigl((\mathcal P c)(t)\bigr)^{a-1}\dd t,
\end{equation}
where $\lambda_{p,q}
:=
\frac{A_j(q)-A_j(p)}{q-p}$.

We claim that
\[
 A_j(\zeta)-\lambda_{p,q}\zeta>0\qquad(-1\le \zeta\le1).
\]
For $j=0$, this is immediate because $A_0\equiv1$ and
$\lambda_{p,q}=0$.
For $j\ge1$, the mean value theorem
gives, for some $\xi\in(p,q)$,
\[
 \lambda_{p,q}=A_j'(\xi)=-jJ_j(\arccos\xi),
 \qquad -jJ_j(\pi)<\lambda_{p,q}<0.
\]
If $\zeta=0$, then $A_j(0)=1$; if $\zeta>0$, then
$A_j(\zeta)-\lambda_{p,q}\zeta>0$ because $A_j(\zeta)\ge0$ and
$-\lambda_{p,q}\zeta>0$.
If $\zeta<0$, \eqref{eq:Aj-over-c-derivative} gives
\[
 \frac{A_j(\zeta)}{\zeta}
 \le \frac{A_j(-1)}{-1}
 =-jJ_j(\pi)<\lambda_{p,q}.
\]
Multiplication by $\zeta<0$ gives
$A_j(\zeta)-\lambda_{p,q}\zeta>0$, proving the claim.

The right-hand side of
\eqref{eq:strict-fragment-first-variation} is therefore strictly
positive.  Hence $\mathscr E(c_h)>\mathscr E(c)$ for all sufficiently
small $h>0$.  Since $c_h$ is measurable, takes values in $[-1,1]$,
and satisfies
\[
 \mathcal P c_h\ge\mathcal P c\ge0,
\]
the inequality \eqref{eq:primitive-rearrangement}, applied with $c_h$
in place of $c$, yields
\[
 \mathscr E(c_h)\le\mathscr E(c_h^\downarrow).
\]
Moreover, the fragment interchange is measure preserving, so $c_h$ is
equimeasurable with $c$ and hence
\[
 c_h^\downarrow=c^\downarrow
 \qquad\text{almost everywhere}.
\]
Consequently,
\[
 \mathscr E(c^\downarrow)=\mathscr E(c_h^\downarrow)
 \ge\mathscr E(c_h)>\mathscr E(c).
\]
Thus equality in \eqref{eq:primitive-rearrangement} is impossible
unless $c$ is nonincreasing.
\end{proof}

We now apply the primitive rearrangement inequality to the original and
rearranged weighted profiles.

\begin{proposition}
\label{prop:convex-comparison}
Let $K_k$ be the convex body generated by
\cref{cor:rearranged-C11-closure}.  Then
\begin{align}
  \mathcal M_k(\partial K_k)
  &=
  \mathcal M_k(\Sigma),
  \label{eq:H-preserved}\\
  \mathcal M_\ell(\partial K_k)
  &\ge
  \mathcal M_\ell(\Sigma)
  \qquad(0\le\ell<k).
  \label{eq:lower-order-increases}
\end{align}
Equality in \eqref{eq:lower-order-increases} for any fixed $\ell<k$
holds if and only if the original angle $\theta$ is nondecreasing,
equivalently, if and only if $\Sigma$ is convex.
\end{proposition}

\begin{proof}
By \cref{cor:rearranged-C11-closure}, the unified profile formula
\eqref{eq:unified-Mj-profile} applies to $\partial K_k$, while
\cref{prop:C11-Steiner-compatibility} identifies the resulting
curvature functionals with the classical quermassintegrals of $K_k$.

Applying \eqref{eq:unified-Mj-profile} to the $C^{1,1}$ arclength
profile $(\widetilde x,\widetilde z)$ and changing variables
$\sigma=S(\tau)$ give
\[
\begin{aligned}
	\mathcal M_k(\partial K_k)
	=
	\frac{\sphereareanm(n-k)}{n}
	\int_0^{L^*}F_k(\widetilde\vartheta)
	\widetilde x^{m-1}\dd\sigma
	=
	\frac{\sphereareanm(n-k)}{n}
	\int_0^T F_k(\vartheta^\uparrow)\dd\tau.
\end{aligned}
\]
By the functional invariance
\eqref{eq:rearrangement-functional-invariance},
\[
\int_0^T F_k(\vartheta^\uparrow)\dd\tau
=
\int_0^T F_k(\vartheta)\dd\tau.
\]
Therefore, by \eqref{eq:body-tau},
\[
	\mathcal M_k(\partial K_k)
	=\mathcal M_k(\Sigma).
\]
This proves \eqref{eq:H-preserved}.

We now fix $0\le\ell<k$.  With the notation of
\cref{prop:primitive-rearrangement}, make the identifications
\[
 c=\cos\vartheta,
 \qquad
 \mathcal P c=Y,
 \qquad
 c^\downarrow=\cos\vartheta^\uparrow,
 \qquad
 \mathcal P(c^\downarrow)=Y^*,
 \qquad
 a=\frac{k-\ell}{m}>0.
\]
Because both $\vartheta$ and
$\vartheta^\uparrow$ take
values in $[0,\pi]$, one also has
\[
 A_\ell(c)=F_\ell(\vartheta),
 \qquad
 A_\ell(c^\downarrow)=F_\ell(\vartheta^\uparrow).
\]
The unified profile formula \eqref{eq:unified-Mj-profile} and
$\dd\tau=x(s)^{m-1}\dd s$ give
\[
 x(s)^{n-\ell-1}\dd s
 =X(\tau)^{k-\ell}\dd\tau.
\]
Since $X^m=mY$, we also have $X^{k-\ell}=m^aY^a$.  Consequently
\begin{align}
 \mathcal M_\ell(\Sigma)
 &=\frac{\sphereareanm(n-\ell)}n
   \int_0^T F_\ell(\vartheta)X^{k-\ell}\dd\tau \notag\\
 &=\frac{\sphereareanm(n-\ell)}n m^a
   \int_0^T A_\ell(c)Y^a\dd\tau .
 \label{eq:lower-functional-original}
\end{align}
For the $C^{1,1}$ convex comparison body $K_k$,
\eqref{eq:unified-Mj-profile} and the same change of variables give
\[
 \int_0^{L^*}F_\ell(\widetilde\vartheta)
       \widetilde x^{n-\ell-1}\dd\sigma
 =\int_0^T F_\ell(\vartheta^\uparrow)
       (X^*)^{k-\ell}\dd\tau.
\]
Finally, $(X^*)^m=mY^*$.  The preceding $C^{1,1}$ profile formula
therefore gives
\begin{align}
 \mathcal M_\ell(\partial K_k)
 &=\frac{\sphereareanm(n-\ell)}n
   \int_0^T F_\ell(\vartheta^\uparrow)(X^*)^{k-\ell}\dd\tau \notag\\
 &=\frac{\sphereareanm(n-\ell)}n m^a
   \int_0^T A_\ell(c^\downarrow)(Y^*)^a\dd\tau .
 \label{eq:lower-functional-rearranged}
\end{align}
The right-hand sides of
\eqref{eq:lower-functional-original} and
\eqref{eq:lower-functional-rearranged} are precisely the two sides of
\eqref{eq:primitive-rearrangement}, multiplied by the same positive
factor.
Hence \cref{prop:primitive-rearrangement} yields
\eqref{eq:lower-order-increases}. 

Finally, $\mathcal P c=Y>0$ on $(0,T)$.
The equality statement in \cref{prop:primitive-rearrangement} therefore
shows that equality holds exactly
when $c$ is nonincreasing.  Since $s\mapsto\tau(s)$ is strictly
increasing and cosine is strictly decreasing on $[0,\pi]$, this is
equivalent to $\theta$ being nondecreasing.  The curvature criterion
used in \cref{cor:rearranged-C11-closure} then gives convexity of
$\Sigma$.  Conversely, convexity yields
$\dot\theta=\kappa_1\ge0$ almost everywhere; since $\theta$ is
Lipschitz, it is nondecreasing.
\end{proof}

\section{Proof of the main inequalities}\label{sec:proofs}

With these preliminaries, the comparison principle reduces the 
axiconvex inequality to the classical Alexandrov--Fenchel inequality for
a convex body.  

\begin{proof}[Proof of \cref{thm:closed-high-dim}]
Fix \(0\le\ell<k\le n-1\). It follows from \cref{prop:convex-comparison} that there exists a convex body \(K_k\) such that
\[
 \mathcal M_k(\partial K_k)=\mathcal M_k(\Sigma),
 \qquad
 \mathcal M_\ell(\partial K_k)\ge\mathcal M_\ell(\Sigma).
\]
Applying the classical convex Alexandrov--Fenchel inequality
\eqref{eq:convex-AF-prelim} to \(K_k\) gives
\[
 \left(\frac{\mathcal M_\ell(\Sigma)}{|\Sn^n|}\right)^{1/(n-\ell)}
 \le
 \left(\frac{\mathcal M_\ell(\partial K_k)}{|\Sn^n|}\right)^{1/(n-\ell)}
 \le
 \left(\frac{\mathcal M_k(\partial K_k)}{|\Sn^n|}\right)^{1/(n-k)}
 =
 \left(\frac{\mathcal M_k(\Sigma)}{|\Sn^n|}\right)^{1/(n-k)}.
\]

If equality holds, then 
$
\mathcal M_\ell(\partial K_k)
=
\mathcal M_\ell(\Sigma).
$
So \cref{prop:convex-comparison} makes \(\theta\)
nondecreasing and the original domain convex.  Applying
\cref{prop:C11-Steiner-compatibility} gives
\[
 \mathcal M_j(\Sigma)=(n+1)W_{j+1}(\Omega).
\]
The assumed equality is therefore equality in the classical convex
Alexandrov--Fenchel inequality for $\Omega$.  Its rigidity statement then forces
$\Omega$ to be a ball. 
\end{proof}

\begin{proof}[Proof of \cref{cor:complete-W-chain}]
Suppose first that $1\le i<j\le n$ and put
$\ell=i-1$, $k=j-1$.  Since
$|\Sn^n|=(n+1)\omega_{n+1}$ and
$\mathcal M_\ell=(n+1)W_i$, the inequality \eqref{eq:axiconvex-curvature-inequality} is exactly
\begin{equation}\label{eq:W-chain-from-W1}
	 \left(\frac{W_i(\Omega)}{\omega_{n+1}}\right)^{1/(n+1-i)}
	\le
	\left(\frac{W_j(\Omega)}{\omega_{n+1}}\right)^{1/(n+1-j)}.
\end{equation}
For $i=0$, the Euclidean isoperimetric inequality is
\[
 \left(\frac{W_0(\Omega)}{\omega_{n+1}}\right)^{1/(n+1)}
 \le
 \left(\frac{W_1(\Omega)}{\omega_{n+1}}\right)^{1/n}.
\]
Combining it with \eqref{eq:W-chain-from-W1} proves \eqref{eq:W-chain}.  If $i\ge1$, equality forces a ball by
\cref{thm:closed-high-dim}.  If $i=0$, equality holds in the
isoperimetric inequality, and again $\Omega$ is a ball. 
\end{proof}

To treat profiles whose tangent angle may leave the interval $[0,\pi]$, we fold the angle
back into $[0,\pi]$ as in \cite[Section~6]{DHMT}. 

\begin{lemma}
\label{lem:rotational-fold}
Let $\Sigma^n\subset\R^{n+1}$ be admissible in the sense of
\cref{def:admissible-axisymmetric}, with profile $\gamma=(x,z)$
and tangent angle $\theta$, and set
\[
 \bar\gamma(s)=(\bar x(s),\bar z(s))
 :=\left(\int_0^s\cos\theta(t)\dd t,\int_0^s|\sin\theta(t)|\dd t\right),
 \qquad 0\le s\le L.
\]
Then $\bar\gamma$ is an admissible $C^{1,1}$ profile generating an
axiconvex hypersurface $\bar\Sigma$ and
\begin{equation}\label{eq:folded-functional-identification}
 \mathcal M_j(\bar\Sigma)
 =\widehat{\mathcal M}_j(\Sigma)
 \qquad(0\le j\le n-1).
\end{equation}
Moreover, \(\bar\Sigma\) is a round sphere if and only if \(\Sigma\) is
a round sphere.
\end{lemma}

\begin{proof}
Set $\bar\theta:=\arccos(\cos \theta)$.  The map
$u\mapsto\arccos(\cos u)$ is the $2\pi$-periodic triangular wave that
agrees with $|u|$ on $[-\pi,\pi]$, hence it is globally $1$-Lipschitz.  Therefore
$\bar{\theta}$ is Lipschitz continuous and takes values in \([0,\pi]\). It satisfies 
\[
\cos\bar\theta=\cos\theta,\qquad
\sin\bar\theta=|\sin\theta|,\qquad
\bar\theta(0)=0,\qquad \bar\theta(L)=\pi.
\]
By definition, for every $s\in[0,L]$,
\begin{equation*}
	\begin{aligned}
		\bar{x}(s)=\int_0^s\cos\bar{\theta}(t)\dd t=x(s),\qquad
		\bar{z}(s)=\int_0^s\sin\bar{\theta}(t)\dd t\ge z(s).
	\end{aligned}
\end{equation*}
Hence, we have
\begin{equation*}
	\bar{x}(L)=x(L)=0,\quad \bar{x}(s)=x(s)>0\quad\text{for }0<s<L,\quad \bar{z}(L)\ge z(L)>0.
\end{equation*}
Since $\bar\theta\in C([0,L];[0,\pi])$,
\cite[Proposition~3.1]{DHMT} implies that $\bar\gamma$ is injective.
The Lipschitz continuity of $\bar\theta$ gives $\bar\gamma\in C^{1,1}$.
The endpoint conditions and the radial-graph argument in
\cref{cor:rearranged-C11-closure} give global $C^{1,1}$
regularity of $\bar\Sigma$.
These facts establish admissibility
in the sense of \cref{def:admissible-axisymmetric}.
The generated hypersurface $\bar\Sigma$ is axiconvex because
$0\le\bar\theta\le\pi$.

It remains to identify the curvatures. Both $\dot\theta$ and $\dot{\bar\theta}$ vanish almost
everywhere on $\{\sin\theta=0\}$.
Therefore 
\[
\dot{\bar\theta}=\varepsilon_\theta\dot\theta
\qquad\text{a.e. on }(0,L).
\]
For $1\le j\le n-1$, substituting $\bar x=x$,
$\sin\bar\theta=\varepsilon_\theta\sin\theta$, and
$\dot{\bar\theta}=\varepsilon_\theta\dot\theta$ into
\eqref{eq:H-integrand-prelim} gives directly
\[
 \begin{aligned}
 \mathcal M_j(\bar\Sigma)
 &=\frac{\sphereareanm}{n}\int_0^L
 \bigl[j\dot{\bar\theta}\sin^{j-1}\bar\theta\,x^{n-j}
       +(n-j)\sin^j\bar\theta\,x^{n-j-1}\bigr]\dd s\\
 &=\frac{\sphereareanm}{n}\int_0^L\varepsilon_\theta^j
 \bigl[j\dot\theta\sin^{j-1}\theta\,x^{n-j}
       +(n-j)\sin^j\theta\,x^{n-j-1}\bigr]\dd s\\
 &=\int_\Sigma\varepsilon_\theta^jH_j\dd A
 =\widehat{\mathcal M}_j(\Sigma).
 \end{aligned}
\]
For $j=0$, the area formula \eqref{eq:area-element-prelim} and
$\widehat H_0=1$ give
\[
 \mathcal M_0(\bar\Sigma)
 =\sphereareanm\int_0^L\bar x^{n-1}\dd s
 =\sphereareanm\int_0^L x^{n-1}\dd s
 =\widehat{\mathcal M}_0(\Sigma).
\]
This proves \eqref{eq:folded-functional-identification}.

Suppose now that $\bar\Sigma$ is a round sphere of radius $R$.
Its arclength meridian satisfies
\[
 L=\pi R,\qquad
 \bar\theta(s)=\frac{s}{R},\qquad
 x(s)=R\sin(s/R).
\]
By the folding identity,
\[
 |\sin\theta(s)|=\sin\bar\theta(s)=\sin(s/R)>0,
 \qquad 0<s<L.
\]
Thus the continuous function $\sin\theta$ has constant sign on $(0,L)$.
Since
\[
 0<z(L)=\int_0^L\sin\theta(s)\dd s,
\]
the function $\sin\theta$ is positive, so $\sin\theta(s)=\sin(s/R)$ on $(0,L)$.
Integrating $\dot z=\sin\theta$ gives
\[
 z(s)=\int_0^s\sin(t/R)\dd t
     =R\bigl(1-\cos(s/R)\bigr),
 \qquad 0\le s\le L.
\]
Together with $x(s)=R\sin(s/R)$, this yields
\[
 x(s)^2+\bigl(z(s)-R\bigr)^2=R^2,
 \qquad 0\le s\le L.
\]
Hence the original meridian is a semicircle of radius $R$, and
$\Sigma$ is a round sphere.  Conversely, if $\Sigma$ is a round sphere, its 
tangent angle already lies in $[0,\pi]$. Thus
$|\sin\theta|=\sin\theta$, so folding leaves the meridian unchanged. Therefore $\bar\Sigma$ is also a round sphere.
\end{proof}

The above lemma reduces the problem to an admissible
axiconvex hypersurface, so the inequality in \cref{thm:folded-homogeneous-chain} reduces to its
ordinary curvature functionals.

\begin{proof}[Proof of \cref{thm:folded-homogeneous-chain}]
Let \(\bar\Sigma\) be the folded hypersurface constructed by
\cref{lem:rotational-fold}.  The identification
\eqref{eq:folded-functional-identification} reduces every folded
functional to a curvature functional of the
axiconvex hypersurface \(\bar\Sigma\).
By \eqref{eq:folded-functional-identification} and the positivity following \eqref{eq:unified-Mj-profile},
\[
\widehat{\mathcal M}_j(\Sigma)
=\mathcal M_j(\bar\Sigma)>0,
\qquad 0\le j\le n-1.
\]
For \(0\le\ell<k\le n-1\), the inequality in \cref{thm:closed-high-dim}
and \eqref{eq:folded-functional-identification} give
\[
 \left(
   \frac{\widehat{\mathcal M}_\ell(\Sigma)}{|\Sn^n|}
 \right)^{1/(n-\ell)}
 \le
 \left(
   \frac{\widehat{\mathcal M}_k(\Sigma)}{|\Sn^n|}
 \right)^{1/(n-k)}.
\]

If equality holds in \eqref{eq:folded-complete-chain}, then equality
holds in the corresponding axiconvex inequality for
\(\bar\Sigma\).  Thus \cref{thm:closed-high-dim} makes
\(\bar\Sigma\) a round sphere, and \cref{lem:rotational-fold} then
makes \(\Sigma\) a round sphere. Conversely, for a sphere of radius \(R\), one has
\(\varepsilon_\theta=1\) away from the poles and
$
 \widehat{\mathcal M}_j(\Sigma)
 =|\Sn^n|R^{n-j}.
$
Direct substitution gives equality in
\eqref{eq:folded-complete-chain}.
\end{proof}

For even orders, the curvature identities from the folding lemma remove
the sign correction.

\begin{proof}[Proof of \cref{cor:even-order-chain}]
For $j=0$, we have $\widehat H_0=H_0=1$ by definition. For even $0<j\le n-1$, $\varepsilon_\theta^j=1$ where
$\sin\theta\ne0$. On $\{\sin\theta=0\}$, all azimuthal principal curvatures vanish.
Hence $H_j=0$ almost everywhere on $\{\sin\theta=0\}$.
Thus, for every even $0\le j\le n-1$,
\[
 \widehat H_j=H_j\quad\text{a.e.},\qquad
 \widehat{\mathcal M}_j(\Sigma)=\mathcal M_j(\Sigma).
\]
Using this identity in
\eqref{eq:folded-complete-chain} proves
\eqref{eq:even-order-curvature-chain}.  For even $\ell$ and arbitrary
$k$ with $0\le\ell<k\le n-1$, 
\eqref{eq:folded-complete-chain} and $\widehat H_k\le|H_k|$ give
\[
 \left(\frac{\mathcal M_\ell(\Sigma)}{|\Sn^n|}\right)^{1/(n-\ell)}
 \le
 \left(\frac{\widehat{\mathcal M}_k(\Sigma)}{|\Sn^n|}\right)^{1/(n-k)}
 \le
 \left(\frac{\int_\Sigma|H_k|\dd A}{|\Sn^n|}\right)^{1/(n-k)}.
\]
This proves \eqref{eq:even-absolute-chain}.  Equality in either
\eqref{eq:even-order-curvature-chain} or \eqref{eq:even-absolute-chain}
forces equality in \eqref{eq:folded-complete-chain}, so $\Sigma$ is a
round sphere by \cref{thm:folded-homogeneous-chain}.  Conversely, a round
sphere realizes equality in both comparisons.
\end{proof}

\begin{remark}
The main inequalities of \cref{thm:folded-homogeneous-chain} can be extended from a single connected hypersurface to a finite disjoint union $\Sigma=\bigsqcup_{\alpha=1}^N\Sigma_\alpha$, where each component satisfies \cref{def:admissible-axisymmetric}.
The components may have different rotation axes.  Define the folded functionals
additively:
\[
 \widehat{\mathcal M}_j(\Sigma)
 :=\sum_{\alpha=1}^N\widehat{\mathcal M}_j(\Sigma_\alpha),
 \qquad 0\le j\le n-1.
\]
The inequality \eqref{eq:folded-complete-chain} remains valid.  
Indeed, for
$p=\frac{n-k}{n-\ell}\in(0,1)$,
\begin{align*}
 \widehat{\mathcal M}_k(\Sigma)
 \ge |\Sn^n|^{1-p}
     \sum_{\alpha=1}^N\widehat{\mathcal M}_\ell(\Sigma_\alpha)^p
 \ge |\Sn^n|^{1-p}
     \widehat{\mathcal M}_\ell(\Sigma)^p.
\end{align*}
The second inequality in the above estimate is strict when $N\ge2$,
since every summand is positive.  Thus equality 
holds exactly when
$N=1$ and the single component is a round sphere.
If all components are axiconvex, the folded functionals coincide with
their ordinary curvature integrals.
\end{remark}

\section*{Declarations}

{\bf AI usage.}
The authors used GPT-6.0 Astra and KIMI K3 for exploratory computations, possible proof directions, and editorial polishing.  In particular,
GPT-6.0 Astra assisted in developing the proof of
\cref{prop:primitive-rearrangement}
by suggesting a possible proof strategy. The authors have thoroughly checked all mathematical derivations and proofs and take full responsibility for the entire content of this manuscript.

{\bf Acknowledgements.}
W. Sheng was partially supported by National Key R$\&$D Program of China (No. 2022YFA1005500) and Natural Science Foundation of China under Grant No. 12571063.

\textbf{Conflict of interest.} The authors declare that they have no
conflict of interest.

\textbf{Data availability.} No datasets were generated or analysed
during the current study.



\begin{thebibliography}{99}

\bibitem{AFM}
V.~Agostiniani, M.~Fogagnolo, and L.~Mazzieri,
\emph{Minkowski inequalities via nonlinear potential theory},
Arch. Ration. Mech. Anal. \textbf{244} (2022), no.~1, 51--85.

\bibitem{Alexandrov1937}
A.~D. Alexandrov,
\emph{Zur Theorie der gemischten Volumina von konvexen
K\"orpern. II. Neue Ungleichungen zwischen den gemischten
Volumina und ihre Anwendungen},
Mat. Sb. (N.S.) \textbf{2(44)} (1937), no.~6, 1205--1238
(in Russian).

\bibitem{Alexandrov1938}
A.~D. Alexandrov,
\emph{Zur Theorie der gemischten Volumina von konvexen
K\"orpern. III. Die Erweiterung zweier Lehrs\"atze Minkowskis
\"uber die konvexen Polyeder auf beliebige konvexe Fl\"achen},
Mat. Sb. (N.S.) \textbf{3(45)} (1938), no.~1, 27--46
(in Russian).

\bibitem{Brendle2021}
S.~Brendle,
\emph{The isoperimetric inequality for a minimal submanifold in Euclidean space},
J. Amer. Math. Soc. \textbf{34} (2021), no.~2, 595--603.

\bibitem{ChangWang2013}
S.-Y.~A. Chang and Y.~Wang,
\emph{Inequalities for quermassintegrals on $k$-convex domains},
Adv. Math. \textbf{248} (2013), 335--377.

\bibitem{ChangWang}
S.-Y.~A. Chang and Y.~Wang,
\emph{Some higher order isoperimetric inequalities via the method of optimal transport},
Int. Math. Res. Not. IMRN \textbf{2014} (2014), no.~24, 6619--6644.

\bibitem{CEK}
O.~Chodosh, M.~Eichmair, and T.~Koerber,
\emph{On the Minkowski inequality near the sphere},
Int. Math. Res. Not. IMRN \textbf{2026} (2026), no.~13,
Paper No.~rnag142.

\bibitem{DHMT}
J.~Dalphin, A.~Henrot, S.~Masnou, and T.~Takahashi,
\emph{On the minimization of total mean curvature},
J. Geom. Anal. \textbf{26} (2016), no.~4, 2729--2750.

\bibitem{Federer1959}
H.~Federer,
\emph{Curvature measures},
Trans. Amer. Math. Soc. \textbf{93} (1959), no.~3, 418--491.

\bibitem{FreireSchwartz2014}
A.~Freire and F.~Schwartz,
\emph{Mass-capacity inequalities for conformally flat manifolds with boundary},
Comm. Partial Differential Equations \textbf{39} (2014), no.~1, 98--119.

\bibitem{Gerhardt1990}
C.~Gerhardt,
\emph{Flow of nonconvex hypersurfaces into spheres},
J. Differential Geom. \textbf{32} (1990), no.~1, 299--314.

\bibitem{Glaudo}
F.~Glaudo,
\emph{Minkowski inequality for nearly spherical domains},
Adv. Math. \textbf{408} (2022), Paper No.~108595.

\bibitem{GuanLi}
P.~Guan and J.~Li,
\emph{The quermassintegral inequalities for $k$-convex starshaped domains},
Adv. Math. \textbf{221} (2009), no.~5, 1725--1732.

\bibitem{Huisken2009}
G.~Huisken,
\emph{Inverse mean curvature flow and isoperimetric inequalities},
lecture, Institute for Advanced Study, 2009,
video available at \url{https://video.ias.edu/node/233}.

\bibitem{HuiskenIlmanen2001}
G.~Huisken and T.~Ilmanen,
\emph{The inverse mean curvature flow and the Riemannian Penrose inequality},
J. Differential Geom. \textbf{59} (2001), no.~3, 353--437.

\bibitem{Kawohl}
B.~Kawohl,
\emph{Rearrangements and Convexity of Level Sets in PDE},
Lecture Notes in Mathematics, vol.~1150, Springer, Berlin, 1985.

\bibitem{Kesavan}
S.~Kesavan,
\emph{Symmetrization and Applications},
Series in Analysis, vol.~3, World Scientific, Hackensack, 2006.

\bibitem{LWZ}
H.~Li, Y.~Wan, and Y.~Zhang,
\emph{On the minimization of total $(n-2)$-th mean curvature in $\R^n$},
J. Geom. Anal. \textbf{36} (2026), Paper No.~363.

\bibitem{ManiLevitska1993}
P.~Mani-Levitska,
\emph{Characterizations of convex sets},
in Handbook of Convex Geometry, Vol.~A,
P.~M.~Gruber and J.~M.~Wills (eds.),
North-Holland, Amsterdam, 1993, pp.~19--41.

\bibitem{MichaelSimon}
J.~H. Michael and L.~M. Simon,
\emph{Sobolev and mean-value inequalities on generalized submanifolds of $\R^n$},
Comm. Pure Appl. Math. \textbf{26} (1973), no.~3, 361--379.

\bibitem{NobiliViolo}
F.~Nobili and I.~Y. Violo,
\emph{Fine P\'olya--Szeg\H{o} rearrangement inequalities in metric spaces and applications},
Calc. Var. Partial Differential Equations \textbf{64} (2025),
no.~9, Paper No.~276.

\bibitem{Qiu}
G.~Qiu,
\emph{A family of higher-order isoperimetric inequalities},
Commun. Contemp. Math. \textbf{17} (2015), no.~3,
Paper No.~1450015, 20 pp.

\bibitem{Schneider}
R.~Schneider,
\emph{Convex Bodies: The Brunn--Minkowski Theory},
2nd expanded ed., Encyclopedia of Mathematics and its Applications,
vol.~151, Cambridge University Press, Cambridge, 2014.

\bibitem{Scholtes}
S.~Scholtes,
\emph{On hypersurfaces of positive reach, alternating Steiner formulae
and Hadwiger's problem},
arXiv:1304.4179, 2013.

\bibitem{Urbas1990}
J.~I.~E. Urbas,
\emph{On the expansion of starshaped hypersurfaces by symmetric
functions of their principal curvatures},
Math. Z. \textbf{205} (1990), no.~3, 355--372.

\bibitem{WangYang}
Y.~Wang and S.~Yang,
\emph{Stability of Minkowski inequality for nearly spherical sets},
arXiv:2511.21943, 2025.

\end{thebibliography}
\end{document}